\documentclass[12pt,twoside]{amsart}
\usepackage{amscd,amsfonts,amsmath,amssymb,amsthm,latexsym}
\usepackage{enumerate,array,srcltx,amsmath}
\usepackage{leqno}
\usepackage[dvips]{graphicx}
\usepackage{caption}
\usepackage{subcaption}

\usepackage{xcolor}

\definecolor{darkblue}{rgb}{0,0,0.545098}
\definecolor{darkgreen}{rgb}{0,0.392157,0}

\usepackage[driverfallback=pdfmx,,colorlinks,bookmarksopen,bookmarksnumbered,citecolor=blue, urlcolor=blue]{hyperref}

\newtheorem{theorem}{Theorem}[section]
\newtheorem{lemma}[theorem]{Lemma}
\newtheorem{proposition}[theorem]{Proposition}
\theoremstyle{definition}

\theoremstyle{remark}
\newtheorem{remark}[theorem]{Remark}
\numberwithin{equation}{section}
\newcommand{\T}{\mathbb{T}}
\newcommand{\I}{\mathbb{I}}
\newcommand{\D}{\mathbb{D}}

\newcommand{\R}{\mathbb{R}}

\newcommand{\func}[1]{\operatorname{#1}}
\def\R{\mathbb{R}}

\def\limCo{-\dfrac{1}{2}\left(\dfrac{1}{\tau }-\dfrac{1}{\beta}\right)}
\def\Re{\func{Re}}

\title[MOORE-GIBSON-THOMPSON EQUATION IN $\mathbb{R}^N$]{Wellposedness and  decay rates for the Cauchy problem of the Moore-Gibson-Thompson equation arising in high intensity ultrasound}
\author{M. Pellicer$^{1}$}
\thanks{$^{1}$Dpt. d'Inform\`atica, Matem\`atica Aplicada i Estad\'{\i}stica (EPS), Universitat de Girona, Campus de Montilivi, 17071 Girona, Catalunya (Spain), e-mail: martap@imae.udg.edu}
\author{B. Said-Houari$^{2}$}
\thanks{$^{2}$Mathematics and Natural Sciences Department, Alhosn University,
P.O. Box: 38772, Abu Dhabi (United Arab Emirates),
e-mail: bsaidhouari@gmail.com}
\makeatother

\date{\today}

\begin{document}


\maketitle

\begin{abstract}
In this paper, we study the Moore--Gibson--Thompson equation in $\R^N$, which is a third order in time equation that arises in viscous thermally relaxing fluids and also in viscoelastic materials (then under the name of \emph{standard linear viscoelastic} model). First, we use some Lyapunov functionals in the Fourier space to show that, under certain assumptions on some parameters in the equation, an energy norm related with the solution decays with a rate $(1+t)^{-N/4}$. But this does not give the decay rate of the solution itself. Hence, in the second part of the paper, we show an explicit representation of the solution in the frequency domain by analyzing the eigenvalues of the  Fourier image of the solution and writing the solution accordingly. We use this eigenvalues expansion method to give the decay rate of the solution (and its derivatives), which (for the solution) results in $(1+t)^{1-N/4}$ for $N=1,2$ and $(1+t)^{1/2-N/4}$ when $N\geq 3$.
\end{abstract}


\section{Introduction, derivation of the model and well-posedness.}
Acoustic is an active field of research which is concerned with the generation and space-time evolution of small mechanical perturbations in fluid (sound waves) or in solid (elastic waves). One of the important equations in nonlinear acoustics is
the  Kuznetsov equation:
\begin{equation}\label{Kuznetsov_1}
u_{tt}-c^{2}\Delta u-b \Delta u_{t}=\frac{\partial}{\partial t}\left(\frac{1}{c^2}\frac{B}{2A}(u_t)^2+|\nabla u|^2\right),
\end{equation}
where $u$ represents the acoustic velocity potential, and $b$, $c$, and ${B}/{A}$ are the diffusivity and speed of sound, and the nonlinearity parameter, respectively. The derivation of equation \eqref{Kuznetsov_1} can be obtained from the general equations of fluid mechanics but we include a brief summary of its derivation, that can be found with more details in \cite{Kaltenbacher_2011}, \cite{Treeby_2012} or \cite{Coulouvrat_1992} and the references therein, for instance.  It is based on the relations coming from the conservation of mass, conservation of momentum and entropy balance in the model of thermo-viscous flow in a compressible fluid:
\begin{itemize}
\item  the equation of conservation of mass (continuity equation):
\begin{equation}\label{Mass_Conser}
\varrho_t+\nabla\cdot (\varrho v)=0,
\end{equation}
\item the equation of conservation of momentum (Newton's second law):
\begin{equation}\label{Conservation_Momuntum}
 \varrho(v_t+(v\cdot \nabla)v )=\nabla\cdot \T,
\end{equation}
\item the equation of conservation of energy (first law of thermodynamics):

\begin{equation}\label{Energy_Equation}
\varrho \theta(\eta_t+(v\cdot \nabla) \eta)=-\nabla\cdot q+\T:\D.
\end{equation}
\end{itemize}
In the previous equations we are considering $v$, $p$ and $\eta$ as the velocity, pressure and specific entropy of the acoustic particle, and $\varrho$ and $q$ as mass density and the heat flux vector, respectively.  Also,  $\D$ is the deformation or strain tensor given by $$\D=\frac{1}{2}(\nabla v+(\nabla v)^T)$$ and  $\T$ is the Cauchy--Poisson stress
tensor  given by $$\T=(-p+\lambda (\nabla\cdot v))\I +2\mu\D,$$
 where $\I$ is the identity matrix,   $\mu$ is the shear viscosity (the first coefficient of viscosity) and
  $\lambda=\zeta-\frac{2}{3}\mu$, where $\zeta$ is the second coefficient of viscosity (the bulk viscosity). The components of $\T:\D$ are $T_{ij}D_{ij}$, where $T_{ij}, D_{ij}$ are the components of $\T$ and $\D$, respectively.

It can be seen (see the above references for the details) that equations \eqref{Conservation_Momuntum} and \eqref{Energy_Equation} can be rewritten as
\begin{equation}\label{Conservation_Momuntum_2}
 \varrho(v_t+(v\cdot \nabla)v)=-\nabla p+\mu \Delta v+(\zeta+\mu/3)\nabla(\nabla\cdot v)
\end{equation}
and
\begin{equation}\label{Energy_Equation_2}
 \varrho\theta (\eta_t+(v\cdot \nabla) \eta)=2\mu\D:\D+\lambda (\nabla\cdot v)^2-\nabla\cdot q,
\end{equation}
respectively.

The previous equations, together with \eqref{Mass_Conser} and the following equations of state
\begin{equation}\label{Equation_State_0}
p=p(\varrho,\eta), \qquad \theta=\theta(\varrho,\eta)
\end{equation}
are the Navier Stokes equations.

First, we assume that the deviation of $\varrho,\, p,\,\eta$ and $\theta$ from their equilibrium values  $\varrho_0,\, p_0,\,\eta_0$ and $\theta_0$ is small. By taking the Taylor series expansion  of \eqref{Equation_State_0} around values at rest $\varrho_0$ and $\eta_0$ and ignoring the higher order terms, we get
\begin{equation*}
p(\varrho,\eta)=p(\varrho_0,\eta_0)+\left(\frac{\partial p}{\partial \varrho}(\varrho_0,\eta_0)\right)(\varrho-\varrho_0)+\frac{1}{2}\left(\frac{\partial ^2 p}{\partial \varrho^2}(\varrho_0,\eta_0)\right)(\varrho-\varrho_0)^2+\left(\frac{\partial p}{\partial \eta}(\varrho_0,\eta_0)\right)(\eta-\eta_0).
\end{equation*}
We put
 \begin{equation*}
p_0=p(\varrho_0,\eta_0),\quad A=\varrho_0\frac{\partial p}{\partial \varrho}(\varrho_0,\eta_0)=\varrho_0c^2,\quad B=\varrho_0^2\frac{\partial ^2 p}{\partial \varrho^2}(\varrho_0,\eta_0),\quad \varrho_0\frac{\gamma-1}{\chi }=\frac{\partial p}{\partial \eta}(\varrho_0,\eta_0),
\end{equation*}
where $\nabla p_0=0$,  $\chi$ is the coefficient of volume expansion and $\gamma=c_p/c_v$ is the ratio of specific heat, as $c_p$ and $c_v$ are the specific heat capacities at constant pressure and
constant volume. Then, the pressure $p$ is given by
 \begin{equation}\label{Equation_Pressure}
p(\varrho,\eta)=p_0+\varrho_0c^2\left[ \frac{\varrho-\varrho_0}{\varrho_0}+\frac{B}{2A}\left(\frac{\varrho-\varrho_0}{\varrho_0}\right)^2+\frac{\gamma-1}{\chi c^2}(\eta-\eta_0)\right],
\end{equation}


By  assuming  that the flow is rotation free, that is $\nabla\times v=0$, and introducing the acoustic velocity potential $v=-\nabla u$, then
it has been shown in \cite{kuznetso_1971} and \cite{Coulouvrat_1992}, that equation \eqref{Kuznetsov_1}, can be derived from the above set of equations by assuming the Fourier law of heat conduction
\begin{equation}\label{Fourier_Law}
q=-K\nabla \theta,
\end{equation}
where $K$ is the thermal conductivity and $\theta$ is the absolute temperature.
It is known that by modeling heat conduction with the  Fourier law (\ref{Fourier_Law}), which assumes the
flux $q$ to be proportional to the gradient of the temperature $\nabla
\theta $ at the same time $t$, leads to the
paradox of infinite heat propagation speed  (that is, any thermal disturbance
at a single point has an instantaneous effect everywhere in the medium) and also fails when $q$ increases or $\nabla\theta$ decreases (see \cite{Jordan_2014_1}). To overcome this
drawback, a number of modifications of the basic assumption on the relation
between the heat flux and the temperature have been made, such as
the  Maxwell--Cattaneo law, the Gurtin--Pipkin theory, the Jeffreys law, the Green--Naghdi theory and
others. The common feature of these theories is that all  permit transmission of heat flow as thermal waves
at finite speed. See \cite{Ch98,JLP89} for more details. One of these laws is the  Maxwell--Cattaneo law, that assumes the following relation between heat flux and temperature:
\begin{equation}\label{Cattaneo_Law}
\tau q_t+q=-K\nabla \theta,
\end{equation}
where $\tau$ is the relaxation time of the heat flux (usually small). By considering \eqref{Cattaneo_Law}, instead of  \eqref{Fourier_Law} and combining it with the equations of fluid mechanics, we get, instead of \eqref{Kuznetsov_1}, the Jordan--Moore--Gibson--Thompson equation (see \cite{Jordan_2014_1})

\begin{equation}\label{MGT_1}
\tau u_{ttt}+u_{tt}-c^{2}\Delta u-b \Delta u_{t}=\frac{\partial}{\partial t}\left(\frac{1}{c^2}\frac{B}{2A}(u_t)^2+|\nabla u|^2\right),
\end{equation}
where $b=\delta+\tau c^2$, with $\delta$ being the diffusivity of sound.

In this paper we consider the linearized version of equation \eqref{MGT_1}, known as the Moore--Gibson--Thompson equation in the acoustics theory:
\begin{equation}\label{MGT_2}
\tau u_{ttt}+u_{tt}-c^{2}\Delta u-b \Delta u_{t}=0.
\end{equation}
This linear equation is an active field of research. It also arises in viscous thermally relaxing fluids and has applications in medical and industrial use of high intensity ultrasound such as lithotripsy, thermotherapy or ultrasound cleaning (see \cite{Lasiecka_2015_2}). But it also appears in viscoelasticity theory under the name of \emph{standard linear model} of vicoelasticity (sometimes also called Kelvin or Zener model) to explain the behaviour of certain viscoelastic materials (that is, that exhibit both a viscous fluid and an elastic solid response) such as, for instance, fluids with complex microestructure (see \cite{Lasiecka_2015_1}). In this context, $u$ represents the linear deformations of a viscoelastic solid with an approach that is considered to be more realistic than the usual Kelvin-Voigt model. Actually, this model seems to be the simplest one that reflects both creeping and stress relaxation effects in viscoelastic materials (see \cite{Fung} or \cite{Rabotnov} for more details).

The derivation of equation \eqref{MGT_2} in $\R$ in the context of viscoelasticity theory can be obtained in the following way.  Let us recall that in viscoelasticity theory, springs and dashpots represent the elastic and viscous components of the materials, respectively. According to \cite{Gorain2010}, \cite{MR2013} or \cite{Fung}, among others, in the one dimensional case this equation represents a linear spring connected in series with a Kelvin-Voigt system, that is, another linear spring connected in parallel with a dashpot. This is a common way of approaching viscoelastic systems using a rheological point of view. Using this formulation (see again \cite{Gorain2010}, \cite{MR2013} or \cite{Fung} for more details), it is easy to see that such a system is governed by the following relation
 \begin{equation}\label{Constitutive_Equation}
\sigma+\tau\sigma^{\prime}=E(e+\beta e^\prime),
\end{equation}
where $\sigma$ is the stress and $e$ is the strain. According to \cite{Fung}, $\tau$ is the stress relaxation time under constant strain, $\beta$ is the strain relaxation time under constant stress and $E$ is the relaxed elastic modulus. In any case, all of them are parameters involving the elastic and viscous constants of the material. More concretely, if $\eta$ stands for the dashpot viscosity coefficient, and $E_1,E_2$ represent the Young modulus of the first and second elastic springs respectively, one has that
\begin{equation}\label{parameters1}
\tau=\dfrac{\eta}{E_1+E_2},\qquad E=\dfrac{E_1E_2}{E_1+E_2}, \qquad \beta=\dfrac{\eta}{E_2}.
\end{equation}
However, there are a few references in which the standard linear model is described as a linear spring connected in parallel with a Maxwell model, that is, a spring and a dashpot connected in series (see for instance \cite{Ottonsen}, where this model is also called the 3-parameter model). In this case, the relation between the stress and strain of the system is also \eqref{Constitutive_Equation}, but the parameters are
\begin{equation}\label{parameters2}
\tau=\dfrac{\eta}{E_2},\qquad E=E_1, \qquad \beta=\dfrac{\eta}{E_1}+\dfrac{\eta}{E_2}.
\end{equation}
In both descriptions of the standard viscoelastic model, as relation \eqref{Constitutive_Equation} is the same, the corresponding equation would be
 \begin{equation*}\label{SLV1}
\tau u_{ttt}+u_{tt}-\dfrac{E}{\rho}( u_{xx}+\beta u_{txx})=0,
\end{equation*}
where $\rho$ represents the longitudinal density of the material. This equation is obtained thinking our material as a sequence of increasingly series-coupled systems of \eqref{Constitutive_Equation} type, that is, a continuous model with \eqref{Constitutive_Equation} as single component (see Chapter 6 of \cite{Davis} or Section 2 of \cite{PSM2004} for a similar deduction on different models).

In \cite{Gorain2010}, \cite{MR2013} or \cite{P-SM-2015} it is assumed that $0<\tau<\beta$ (dissipative system), with $\tau,\beta$ being small constants. Observe that in both descriptions of the parameters \eqref{parameters1} and \eqref{parameters2} this is a natural assumption: in \eqref{parameters1} it is fulfilled unless $\eta=0$ (no dashpot), $E_1=0$ (only the Kelvin-Voigt sytem) or $E_2=\infty$ (infinitely rigid second spring) ,while in \eqref{parameters1} it is fulfilled unless $\eta=0$ (no dashpot) or $E_1=\infty$ (infinitely rigid first spring). In all these particular cases, we would obtain the conservative case $\tau=\beta$. The case $\tau>\beta$ is treated in \cite{Conejero}, where the authors show the chaotic behaviour of the corresponding solution.

The initial boundary value problem associated to \eqref{MGT_2} has been studied recently by many authors in bounded domains. In \cite{Kaltenbacher_2011} (see also \cite{Kaltenbacher_2012}), the authors considered the linearized equation
\begin{equation}\label{MGT_22}
\tau u_{ttt}+\alpha u_{tt}+c^{2}\mathcal{A} u+b \mathcal{A} u_{t}=0
\end{equation}
where $\mathcal{A}$ is a positive self-adjoint operator, and  showed  that by
neglecting diffusivity of the sound coefficient ($b=0$) there arises a lack of
existence of a semigroup associated with the linear dynamics. On the other hand, they showed that when the diffusivity of the
sound is strictly positive ($b>0$), the linear dynamics are described by a strongly
continuous semigroup, which is exponentially stable provided that $\gamma=\alpha-\tau c^2/b>0$, while if $\gamma=0$ the energy is conserved (the same type of results are obtained in \cite{MR2013} or \cite{Gorain2010} using energy methods, or in \cite{Trigg_et_al} using the analysis of the spectrum of the operator). The exponential decay rate results in \cite{Trigg_et_al} are completed in \cite{P-SM-2015}, where the obtention of an explicit scalar product where the operator is normal allows the authors to obtain the optimal exponential decay rate of the solutions. Finally, in \cite{Conejero}, the authors show the caotic behaviour of the system when $\gamma>0$, as we have mentioned above.

Observe that in this third order in time equation, the strong damping term $b\mathcal{A}u_t$ is responsible for the well-posedness of the problem, while in the wave equation with strong damping ($\tau=0$) the strong damping term is responsible for the analiticity of the semigroup.


We also mention the recent paper  \cite{Lasiecka_2015_1} where the authors consider \eqref{MGT_22} with a memory damping term and show an exponential decay of the energy provided that the kernel is exponentially decaying. This result is generalized in \cite{Lasiecka_2015_2}, where it is shown that the memory kernel decay determines the solution decay.

To the best of our knowledge, these equations have not been studied yet in an unbounded domain. So, the goal of this paper is to show the well-posedness and investigate the decay rate of the solutions of the Moore--Gibson--Thompson equation in an unbounded domain. Namely,
we consider the equation
\begin{equation}\label{Bose_Problem}
\tau u_{ttt}+u_{tt}-c^{2}\Delta  u-c^{2}\beta \Delta u_{t}=0  \quad \textrm{ in }\mathbb{R}^{N},\quad t>0,
\end{equation}
with the following initial data
\begin{equation}
u\left( x,0\right) =u_{0}\left( x\right) ,\quad u_{t}\left( x,0\right)
=u_{1}\left( x\right) ,\quad u_{tt}\left( x,0\right) =u_{2}\left( x\right).
\label{Initial_data_1}
\end{equation}%

In order to state and prove our results, let us first and
without loss of generality, take  $c=1$. In addition, we assume that $0<\tau <\beta$, which corresponds to the dissipative case.

We first observe the well-posedness of \eqref{Bose_Problem} with initial data in $H^s(\R^N)$, $s\geq 1$. This can be seen using the result of \cite{Kaltenbacher_2011} (also \cite{Trigg_et_al}) after applying the following change of variables to the equation \eqref{Bose_Problem}:
$$ u=e^{-\gamma t}\,v$$
with $\gamma>0$. After this change of variables we obtain the equation
\begin{equation}\label{eq:change}
\tau v_{ttt}+\gamma_0 v_{tt}-\gamma_1\Delta  v-\gamma_2 \Delta v_{t} + \gamma_3 v + \gamma_4 v_{t}=0  \quad \textrm{ in }\mathbb{R}^{N},\quad t>0,
\end{equation}
with
$$\gamma_0=1-3\gamma\tau,\
\gamma_1= c^2(1-\beta\gamma),\
\gamma_2=c^2\beta,\ \gamma_3=\gamma^2(1-\gamma\tau),\
\gamma_4= -\gamma(2-3\gamma\tau).$$
Observe that $\gamma_i>0$ for $i=1,2, 3$ if $\gamma$ is small enough. After rearranging terms, equation \eqref{eq:change} can be written as
\begin{equation}\label{eq:wellposedness}
\tau v_{ttt} + \gamma_0 v_{tt} -\gamma_1\left( \Delta v-\dfrac{\gamma_3}{\gamma_1} v\right)-\gamma_2\left( \Delta v-\frac{\gamma_3}{\gamma_1} v\right) +\left( \gamma_4-\dfrac{\gamma_2\gamma_3}{\gamma_1}\right) v_t =0.
\end{equation}
Observe that, if $\gamma$ is small enough, $\mathcal{A}= -\left( \Delta -\dfrac{\gamma_3}{\gamma_1} Id \right)$ is a positive and self-adjoint operator in $H^1(\R^N)$ and $\gamma_1,\gamma_2>0$. Observe also  that $\left( \gamma_4-\dfrac{\gamma_2\gamma_3}{\gamma_1}\right) v_t $ can be seen as a bounded perturbation term. So, the hypotheses of \cite{Kaltenbacher_2011} are satisfied and, hence, \eqref{Bose_Problem} is a well-posed problem in $H^1(\R^N)$ (and, by linearity, in $H^s(\R^N)$ for $s\geq 1$).


The two main results we obtain for the decay rate of this equation can be seen in Theorems \ref{Theorem_Main} and \ref{Theorem_Solution_Energy_Space} (and \ref{Theorem_N_3}) below. First, in Section \ref{sec:Energy} and using the energy method in the Fourier space, we show that, if $\tau<\beta$, the energy norm of $V$ and of its higher-order derivatives $ \partial_x^jV$, with $V=(\tau u_{tt}+u_t , \partial_x(\tau u_t+ u),\partial_x u_t)$, decay as
\begin{equation}\label{Estimate_Main_Theorem_0}
\Vert \partial_x^jV(t)\Vert_{L^2}\leq C(1+t)^{-N/4-j/2}\Vert V_0\Vert_{L^1}+e^{-ct}\Vert \partial_x^j V_0\Vert_{L^2},
\end{equation}
for a certain $c>0$ (see Theorem \ref{Theorem_Main}). The decay rate in  \eqref{Estimate_Main_Theorem_0} is a direct consequence of the estimate of the Fourier image $\hat{V}(x,t)$:
\begin{equation*}\label{Main_Estimate_V_Fourier_0}
|\hat{V}(\xi,t)|^2\leq Ce^{-c\rho(\xi)t}|\hat{V}(\xi,0)|^2,\qquad \rho(\xi)=\frac{|\xi|^2}{1+|\xi|^2}
\end{equation*}
which we derive using the Lyapunov functional method (see Proposition \ref{Main_estimate}).  This method is a very powerful tool in proving the decay rate of the energy norm (see \cite{IHK08,Ike04}). But, unfortunately,
it is obvious that the estimate \eqref{Estimate_Main_Theorem_0} does not give us any information about the decay of the solution $u(x,t)$ itself. It only gives us the decay rate of the energy norm which eventually follows the decay rate of the slowest component of the vector $V$. So, to overcome this limitation of the Lyapunov functional method, we use the eigenvalues expansion method, which is based essentially on the behavior of the eigenvalues of the equation in the Fourier space. In Section \ref{sec:eigenvalues} we give a complete description of the solutions of the characteristic equation of the corresponding operator and use it to divide the frequency domain into three main parts (low frequency, middle frequency and high frequency regions) and estimate the Fourier image of the solution in each region. In Section \ref{sec:decay} and using this method, we are able to prove our second main result, which is the following decay rate of the solution of \eqref{Bose_Problem}-\eqref{Initial_data_1} when $0<\tau<\beta$ with initial data in $L^1(\R^N)\cap H^s(\R^N)$, $s\geq 1$ (see Theorem \ref{Theorem_Solution_Energy_Space}):
\begin{eqnarray}\label{Estimate_N_2_0}
\left\Vert \partial _{x}^{j}u\left( t\right) \right\Vert _{L^{2}}&\leq& C(\Vert u_0 \Vert_{L^1}+\Vert u_1 \Vert_{L^1}+\Vert u_2\Vert_{L^1})(1+t)^{1-N/4-j/2}\notag\\
&&+C(\Vert\partial_x^j  u_0\Vert_{L^2}+\Vert\partial_x^j  u_1\Vert_{L^2}+\Vert \partial_x^j u_2\Vert_{L^2})e^{-ct}
\end{eqnarray}
(for certain $c,C>0$ and $0\leq j\leq s$)  and even better estimates if $N+j\geq 3$ (see Theorem \ref{Theorem_N_3}):
\begin{eqnarray}\label{Estimate_N_3_0}
\left\Vert \partial _{x}^{j}u\left( t\right) \right\Vert _{L^{2}}&\leq& C(\Vert u_0 \Vert_{L^1}+\Vert u_1 \Vert_{L^1}+\Vert u_2\Vert_{L^1})(1+t)^{-(N-2)/4-j/2}\notag\\
&&+C(\Vert\partial_x^j  u_0\Vert_{L^2}+\Vert\partial_x^j  u_1\Vert_{L^2}+\Vert \partial_x^j u_2\Vert_{L^2})e^{-ct}.
\end{eqnarray}
For initial data in the weighted  space $L^{1,1}(\R^N)\cap H^s(\R^N)$, $s\geq 1$, the above estimates will be improved (see Theorem \ref{Theorem_L_1_1_data} for more details).

To summarize, the remaining part of this paper is organized as follows. In Section \ref{sec:Energy} we use the energy method in the Fourier space to build an appropriate Lyapunov functional, that is used to derive the decay rate of the energy norm explained above. Section \ref{sec:eigenvalues} is devoted to the eigenvalues expansion method, that is used in Section \ref{sec:decay} to derive the decay rate of the solution and its spacial derivatives.

%
%
%


\vspace{1cm}


\section{Energy method in the Fourier space}\label{sec:Energy}
In this section, we apply the energy method in the Fourier space to show the decay rate of the energy-norm of $V=(\tau u_{tt}+u_t , \partial_x(\tau u_t+ u),\partial_x u_t)$, where $u(x,t)$ is the solution of \eqref{Bose_Problem}-\eqref{Initial_data_1}.

First, we can write the problem in the Fourier space taking the Fourier transform of equation \eqref{Bose_Problem} and the initial data \eqref{Initial_data_1}. We then obtain the following ODE initial value problem:
\begin{equation}\label{Bose_Problem_Linear_Fourier}
\tau \hat{u}_{ttt}+\hat{u}_{tt}+|\xi|^2\hat{u}+\beta |\xi|^2\hat{u_t}=0
\end{equation}
and
\begin{equation}\label{Initial_data_Fourier}
\hat{u}\left( \xi,0\right) =u_{0}\left( \xi\right) ,\quad \hat{u}_{t}\left( \xi,0\right)
=\hat{u}_{1}\left( \xi\right) ,\quad \hat{u}_{tt}\left( \xi,0\right) =\hat{u}_{2}\left( \xi \right)
\end{equation}
with $\xi\in\R^N$. Introducing the new variables
\begin{equation*}
\hat{v}=\hat{u}_t \qquad \text{ and } \qquad  \hat{w}=\hat{u}_{tt},
\end{equation*}
the previous ODE can be rewritten as the following first order system
\begin{equation}\label{System_New}
\left\{
\begin{array}{ll}
\hat{u}_t=\hat{v},\vspace{0.2cm}\\
\hat{v}_t=\hat{w},\vspace{0.2cm}\\
\hat{w}_t=-\dfrac{|\xi|^2}{\tau }\hat{u}-\dfrac{\beta |\xi|^2}{\tau }\hat{v}-\dfrac{1}{\tau }\hat{w}.
\end{array}
\right.
\end{equation}

We can write the previous system in a matrix form as
\begin{equation}\label{Matrix_Form}
\hat{U}_t(\xi,t)=\Phi(\xi) \hat{U}(\xi,t),
\end{equation}
with the initial data
\begin{equation*}\label{Initial_data}
\hat{U}_0(\xi)=\hat{U}(\xi,0),
\end{equation*}
where $\hat{U}(\xi,t)=(\hat{u}(\xi,t),\hat{v}(\xi,t),\hat{w}(\xi,t))^T$ and
\begin{equation}\label{LA}
\Phi(\xi)= L+|\xi|^2 A =
\left(
\begin{array}{ccc}
0 & 1 & 0 \\
0 & 0 & 1 \\
0 & 0 & -\dfrac{1}{\tau }
\end{array}%
\right)
+ |\xi|^2
\left(
\begin{array}{ccc}
0 & 0 & 0 \\
0 & 0 & 0 \\
-\dfrac{1}{\tau } & -\dfrac{\beta }{\tau } & 0
\end{array}%
\right).
\end{equation}

Now we define the vector  $V=(\tau u_{tt}+u_t , \partial_x (\tau u_t+ u),\partial_x u_t)$. Thus, the pointwise estimate of the Fourier image of $V$ reads as follows.

\begin{proposition}\label{Main_estimate}
Let $\hat{u}$ be the solution of \eqref{Bose_Problem_Linear_Fourier}-\eqref{Initial_data_Fourier}.  Assume that $\tau<\beta $. Then, the Fourier image of the above vector $V$ satisfies the estimate
\begin{equation}\label{Main_Estimate_V_Fourier}
|\hat{V}(\xi,t)|^2\leq Ce^{-c\rho(\xi)t}|\hat{V}(\xi,0)|^2,
\end{equation}
for all $t\geq 0$ and certain $c,C>0$, where
\begin{equation}\label{rho_function}
\rho(\xi)=\frac{|\xi|^2}{1+|\xi|^2}.
\end{equation}

\end{proposition}
The proof of Proposition \ref{Main_estimate} will be given through some lemmas, where a certain Lyapunov functional is obtained and used.
First, we may rewrite system \eqref{System_New} as
\begin{equation}\label{System_New_2}
\left\{
\begin{array}{ll}
\hat{u}_t=\hat{v},\vspace{0.2cm}\\
\hat{v}_t=\hat{w},\vspace{0.2cm}\\
\tau \hat{w}_t=-|\xi|^2\hat{u}-\beta |\xi|^2\hat{v}-\hat{w}.
\end{array}
\right.
\end{equation}
\begin{lemma}\label{Lemma_Energy_1}
The energy functional associated to system \eqref{System_New_2} is
\begin{equation}\label{Energy_Fourier}
\hat{E} (\xi,t)=\frac{1}{2}\left\{|\hat{v}+\tau \hat{w}|^2+\tau (\beta-\tau )|\xi|^2|\hat{v}|^2+|\xi|^2|\hat{u}+\tau \hat{v}|^2\right\}
\end{equation}
and satisfies, for all $t\geq 0$, the identity
\begin{equation}\label{Energy_Indentity}
\frac{d}{dt}\hat{E} (\xi,t)=-(\beta-\tau )|\xi|^2|\hat{v}|^2.
\end{equation}
\end{lemma}
\begin{proof}

Summing up the second and the third equation in \eqref{System_New_2} we get
\begin{equation}\label{Sum_Equations}
(\hat{v}+\tau \hat{w})_t=-|\xi|^2\hat{u}-\beta |\xi|^2\hat{v}.
\end{equation}
Multiplying \eqref{Sum_Equations} by $\bar{\hat{v}}+\tau \bar{\hat{w}}$ and taking the real parts, we obtain,
\begin{equation}\label{w_estimate}
\frac{1}{2}\frac{d}{dt} |\hat{v}+\tau \hat{w}|^2=-\tau |\xi|^2\func{Re}(\hat{u}\bar{\hat{w}})-\beta\tau |\xi|^2\func{Re}(\hat{v}\bar{\hat{w}})-|\xi|^2\func{Re}(\hat{u}\bar{\hat{v}})-\beta|\xi|^2|\hat{v}|^2.
\end{equation}
Next, multiplying the second equation in \eqref{System_New_2} by $\tau (\beta-\tau )\bar{\hat{v}}$ and taking the real part, we get
\begin{equation}\label{v_Estimate}
\frac{1}{2}\tau (\beta-\tau )\frac{d}{dt}|\hat{v}|^2=\tau (\beta-\tau )\func{Re}(\hat{w}\bar{\hat{v}}).
\end{equation}
Now, multiplying the second equation in \eqref{System_New_2} by $\tau $ and adding the result to the first equation, we obtain
\begin{equation}\label{Equation_u_v}
(\hat{u}+\tau \hat{v})_t=\tau \hat{w}+\hat{v}.
\end{equation}
Multiplying \eqref{Equation_u_v} by $\bar{\hat{u}}+\tau \bar{\hat{v}}$ and taking the real parts, we get
\begin{equation}\label{u_v_derivative}
\frac{1}{2}\frac{d}{dt}|\hat{u}+\tau \hat{v}|^2=\tau \func{Re}(\hat{w}\bar{\hat{u}})+\tau ^2\func{Re}(\hat{w}\bar{\hat{v}})+\func{Re}(\hat{v}\bar{\hat{u}})+\tau |\hat{v}|^2.
\end{equation}
Now, computing $|\xi|^2\eqref{u_v_derivative}+|\xi|^2\eqref{v_Estimate}+\eqref{w_estimate}$, we obtain \eqref{Energy_Indentity}, which finishes the proof of Lemma \ref{Lemma_Energy_1}.
\end{proof}

Now, we define the functional $F_1(\xi,t)$ as
\begin{equation}\label{F_1_Functional}
F_1(\xi,t)=\func{Re}\left\{ (\bar{\hat{u}}+\tau \bar{\hat{v}})(\hat{v}+\tau \hat{w})\right\}.
\end{equation}
Then, we have the following lemma.
\begin{lemma}\label{Lemma_F_1}
For any $\epsilon_0>0$, we have
\begin{equation}\label{dF_1_dt}
\frac{d}{dt}F_1(\xi,t)+(1-\epsilon_0)|\xi|^2|\hat{u}+\tau \hat{v}|^2\leq |\hat{v}+\tau \hat{w}|^2+C(\epsilon_0)|\xi|^2|\hat{v}|^2.
\end{equation}
\end{lemma}
\begin{proof}
Multiplying equation \eqref{Sum_Equations} by $\bar{\hat{u}}+\tau \bar{\hat{v}}$ and equation \eqref{Equation_u_v} by $\bar{\hat{v}}+\tau \bar{\hat{w}}$ we get, respectively,
\begin{eqnarray*}
(\hat{v}+\tau \hat{w})_t(\bar{\hat{u}}+\tau \bar{\hat{v}}) &=&(-|\xi|^2\hat{u}-\beta |\xi|^2\hat{v})(\bar{\hat{u}}+\tau \bar{\hat{v}})\\
&=&(-|\xi|^2\hat{u}-\beta |\xi|^2\hat{v}-\tau |\xi|^2\hat{v}+\tau |\xi|^2\hat{v})(\bar{\hat{u}}+\tau \bar{\hat{v}})
\end{eqnarray*}
and
\begin{equation*}
(\hat{u}+\tau \hat{v})_t(\bar{\hat{v}}+\tau \bar{\hat{w}})=(\tau \hat{w}+\hat{v})(\bar{\hat{v}}+\tau \bar{\hat{w}}).
\end{equation*}
Summing up the above two equations and taking the real part, we obtain
\begin{equation*}
\frac{d}{dt}F_1(\xi,t)+|\xi|^2|\hat{u}+\tau \hat{v}|^2-|\hat{v}+\tau \hat{w}|^2=|\xi|^2(\tau -\beta)\func{Re}(\hat{v}(\bar{\hat{u}}+\tau \bar{\hat{v}})).
\end{equation*}
Applying Young's inequality for any $\epsilon_0>0$, we obtain \eqref{dF_1_dt}. This ends the proof of Lemma \ref{Lemma_F_1}.
\end{proof}

Next, we define the functional $F_2(\xi,t)$ as
\begin{equation}\label{F_2_Functional}
F_2(\xi,t)=-\tau \func{Re}( \bar{\hat{v}}(\hat{v}+\tau \hat{w})).
\end{equation}
\begin{lemma}\label{Lemma_F_2}
For any $\epsilon_1,\epsilon_2>0$, we have
\begin{equation}\label{dF_2_dt_1}
\frac{d}{dt}F_2(\xi,t)+(1-\epsilon_1)|\hat{v}+\tau \hat{w}|^2
\leq C(\epsilon_1,\epsilon_2)(1+|\xi|^2)|\hat{v}|^2+\epsilon_2|\xi|^2|\hat{u}+\tau \hat{v}|^2.
\end{equation}

\end{lemma}
\begin{proof}
Multiplying the second equation in \eqref{System_New_2} by $-\tau (\bar{\hat{v}}+\tau \bar{\hat{w}})$ and \eqref{Sum_Equations} by $-\tau \bar{\hat{v}}$,
we obtain, respectively,
\begin{equation*}
-\tau \hat{v}_t(\bar{\hat{v}}+\tau \bar{\hat{w}})=-\tau \hat{w}(\bar{\hat{v}}+\tau \bar{\hat{w}})
\end{equation*}
and
\begin{eqnarray*}
-\tau (\hat{v}+\tau \hat{w})_t\bar{\hat{v}}&=&(\tau |\xi|^2\hat{u}+\beta\tau |\xi|^2\hat{v})\bar{\hat{v}}\\
&=& \Big(\tau |\xi|^2\hat{u}+\tau \beta |\xi|^2\hat{v}+\tau ^2|\xi|^2\hat{v}-\tau ^2|\xi|^2\hat{v}+(\hat{v}+\tau \hat{w})-(\hat{v}+\tau \hat{w})\Big)\bar{\hat{v}}.
\end{eqnarray*}
Summing up the above two equations and taking the real parts, we obtain
\begin{equation*}
\frac{d}{dt}F_2(\xi,t)+|\hat{v}+\tau \hat{w}|^2-\tau (\beta-\tau )|\xi|^2|\hat{v}|^2
=\tau |\xi|^2\func{Re}\left\{(\hat{u}+\tau \hat{v})\bar{\hat{v}}\right\}+\func{Re}\left\{(\hat{v}+\tau \hat{w})\bar{\hat{v}}\right\}.
\end{equation*}
Applying Young's inequality, we obtain the estimate \eqref{dF_2_dt_1} for any $\epsilon_1,\epsilon_2>0$.
\end{proof}
\begin{proof}[Proof of Proposition \ref{Main_estimate}]
We define the Lyapunov functional $L(\xi,t)$ as
\begin{equation}\label{Lyapunov}
L(\xi,t)=\gamma_0\hat{E}(\xi,t)+\frac{|\xi|^2}{1+|\xi|^2}F_1(\xi,t)+\gamma_1\frac{|\xi|^2}{1+|\xi|^2}F_2(\xi,t),
\end{equation}
where $\gamma_0$ and $\gamma_1$ are positive numbers that will be fixed later on.

Taking the derivative of \eqref{Lyapunov} with respect to $t$ and making use of \eqref{Energy_Indentity}, \eqref{dF_1_dt} and \eqref{dF_2_dt_1}, we obtain
\begin{eqnarray}\label{dL_dt_1}
&&\frac{d}{dt}L(\xi,t)+\Big(\gamma_1(1-\epsilon_1)-1\Big)\frac{|\xi|^2}{1+|\xi|^2}|\hat{v}+\tau \hat{w}|^2\notag\\
&&+\Big((1-\epsilon_0)-\gamma_1\epsilon_2\Big)\frac{|\xi|^2}{1+|\xi|^2}(|\xi|^2|\hat{u}+\tau \hat{v}|^2)\notag\\
&&+\Big(\gamma_0(\beta-\tau )-C(\epsilon_0)-\gamma_1C(\epsilon_1,\epsilon_2)\Big)|\xi|^2|\hat{v}|^2\leq 0,
\end{eqnarray}
where we used the fact that $|\xi|^2/(1+|\xi|^2)\leq 1$.
In the above estimate, we can fix our constants in such a way that the previous coefficients are positive. This can be achieved as follows: we pick $\epsilon_0$ and $\epsilon_1$ small enough such that $\epsilon_0<1$ and $\epsilon_1<1$. After that, we take $\gamma_1$ large enough such that
\begin{equation*}
\gamma_1>\frac{1}{1-\epsilon_1}.
\end{equation*}
Once $\gamma_1$ and $\epsilon_0$ are fixed, we select $\epsilon_2$ small enough such that
\begin{equation*}
\epsilon_2<\frac{1-\epsilon_0}{\gamma_1}.
\end{equation*}
Finally, and recalling that $\tau<\beta$, we may choose $\gamma_0$ large enough such that
\begin{equation*}
\gamma_0>\frac{C(\epsilon_0)+\gamma_1C(\epsilon_1,\epsilon_2)}{\beta-\tau }.
\end{equation*}
Consequently, we deduce that there exists a positive constant $\gamma_2$ such that for all $t\geq 0$,
\begin{eqnarray}\label{dL_dt_2}
&&\frac{d}{dt}L(\xi,t)+\gamma_2\frac{|\xi|^2}{1+|\xi|^2}\hat{E}(\xi,t)\leq 0.
\end{eqnarray}

On the other hand, it is not difficult to see that from \eqref{Lyapunov}, \eqref{Energy_Fourier}, \eqref{F_1_Functional} and \eqref{F_2_Functional} and for $\gamma_0$, large enough, that there exists two positive constants $\gamma_3$ and $\gamma_4$ such that
\begin{equation}\label{Equival_E_L}
\gamma_3\hat{E}(\xi,t)\leq L(\xi,t)\leq \gamma_4\hat{E}(\xi,t).
\end{equation}
Combining \eqref{dL_dt_2} and \eqref{Equival_E_L}, we deduce that there exists a positive constant $\gamma_5$ such that for all $t\geq 0$,
\begin{equation}\label{Estimate_L_main}
\frac{d}{dt}L(\xi,t)+\gamma_5\frac{|\xi|^2}{1+|\xi|^2}L(\xi,t)\leq 0.
\end{equation}
A simple application of Gronwall's lemma, leads to the estimate \eqref{Main_Estimate_V_Fourier}, as $L$ and the norm of $\hat{V}$ are equivalent.
\end{proof}

In order to prove our first main result, we also need the first inequality of the following lemma. The rest of it will be used to prove the decay results in the other sections.

\begin{lemma}\label{Lemma_Inequalities}
For all $t\geq 0$ and for all $j\geq 0,c> 0$, the following estimates hold:
\begin{equation}\label{Estimate_Integral}
\int_{|\xi|\leq 1} |\xi|^je^{-c|\xi|^2t}d\xi\leq C(1+t)^{-N/2-j/2},\qquad \text{for}\qquad N\geq 1.
\end{equation}
Also,
\begin{equation}\label{Estimate_Cos_N_3}
\int_{|\xi|\leq 1}|\xi|^je^{-c|\xi|^2t}\left\vert \cos (t|\xi|)\right\vert^2d\xi\leq C(1+t)^{-N/2-j/2},\qquad \text{for}\qquad N\geq 1.
\end{equation}
Moreover,
\begin{equation}\label{Estimate_Sin_N_2}
\int_{|\xi|\leq  1}|\xi|^je^{-c|\xi|^2t}\left\vert \frac{\sin (t|\xi|)}{|\xi|}\right\vert^2d\xi\leq C(1+t)^{2-N/2-j/2},\qquad \text{for}\qquad N\geq 1
\end{equation}
and
\begin{equation}\label{Estimate_Sin_N_3}
\int_{\R^N}|\xi|^je^{-c|\xi|^2t}\left\vert \frac{\sin (t|\xi|)}{|\xi|}\right\vert^2d\xi\leq Ct^{-(N-2)/2-j/2},\qquad \text{if}\qquad j+N\geq 3.
\end{equation}


\end{lemma}

\begin{proof}

First, to prove inequality \eqref{Estimate_Integral} we will first prove that for given $c>0$ and $k\geq 0$, we have
\begin{equation}\label{Integral_Estimate_0}
\int_0^1 r^ke^{-cr^2t}dr\leq C(1+t)^{-(k+1)/2},
\end{equation}
for all $t\geq 0$, where $C$ is a positive constant independent of $t$.
To see this, observe first that for $0\leq t \leq 1$, the estimate \eqref{Integral_Estimate_0} is obvious. On the other hand, for $t\geq 1$, we have
\begin{equation}\label{t_Estimate}
(1+t)\leq 2t.
\end{equation}
Now, by using \eqref{t_Estimate} and the change of variables $z=cr^2t$,
\begin{eqnarray*}
2^{-(k+1)/2}c^{(k+1)/2}(1+t)^{(k+1)/2}\int_0^1 r^ke^{-cr^2t}dr&\leq &c^{(k+1)/2}t^{(k+1)/2}\int_0^1 r^ke^{-cr^2t}dr\\
&=& \int_0^1  (cr^2 t)^{k/2}e^{-cr^2t}(ct)^{1/2} dr\\
&=&\frac{1}{2} \int_0^{c t}  (z)^{k/2}e^{-z}z^{-1/2} dz\\
&\leq& \frac{1}{2} \int_0^{\infty}  z^{(k+1)/2-1}e^{-z}dz\\
&=& \frac{1}{2}\Gamma\left(\frac{k+1}{2}\right)< \infty,
\end{eqnarray*}
where $\Gamma$ is the gamma function. This yields \eqref{Integral_Estimate_0}. Applying the change of variables $r=|\xi|$ and $d\xi = |\xi|^{N-1}dr$ to the left hand side of \eqref{Estimate_Integral} and using \eqref{Integral_Estimate_0}, \eqref{Estimate_Integral} is immediately obtained.

Second,  the estimate \eqref{Estimate_Cos_N_3}  is straightforward: we may just use the fact that $|\cos (t|\xi|)|\leq 1$ and apply \eqref{Estimate_Integral}.

Third, to show \eqref{Estimate_Sin_N_3} we first see that
\begin{equation*}
|\sin (t|\xi|)|\leq t|\xi|
\end{equation*}
for all $|\xi|\geq 0$.
Using this, we have
\begin{eqnarray*}
\int_{|\xi|\leq  1}|\xi|^je^{-c|\xi|^2t}\left\vert \frac{\sin (t|\xi|)}{\xi}\right\vert^2d\xi&\leq& \int_{|\xi|\leq  1}|\xi|^je^{-c|\xi|^2t}t^2d\xi\\
&\leq& (1+t)^2\int_{|\xi|< 1}|\xi|^je^{-c|\xi|^2t}d\xi\\
&\leq& C(1+t)^2\cdot (1+t)^{-N/2-j/2}\\
&=& C(1+t)^{2-N/2-j/2},
\end{eqnarray*}
where we have used \eqref{Estimate_Integral}. This inequality holds true for all $N\geq 1$, in particular for $N=1,2$. But now,  for $N\geq 3$, we can improve the above estimate and get \eqref{Estimate_Sin_N_3}. Indeed, by taking the change of variable $r=|\xi|$ and  $d\xi=|\xi|^{N-1} dr$, we write
\begin{equation*}
\int_{\R^N}|\xi|^je^{-c|\xi|^2t}\left\vert \frac{\sin (t|\xi|)}{\xi}\right\vert^2d\xi=\int_{0}^\infty r^{j+N-1}e^{-cr^2t}\left\vert \frac{\sin (tr)}{r}\right\vert^2dr.
\end{equation*}
Now, we put the new change of variable $\omega=\sqrt{t}r$ and then we get
\begin{eqnarray*}
\int_{0}^\infty r^{j+N-1}e^{-cr^2t}\left\vert \frac{\sin (tr)}{r}\right\vert^2dr&=&\int_{0}^\infty {(\omega/\sqrt{t})}^{j+N-1}e^{-c\omega^2} \left|\sin(\sqrt{t}\omega)\right|^2(t/\omega^2)(d\omega/\sqrt{t} )\\
&=& (\sqrt{t})^{-(j+N-2)}\int_{0}^\infty \omega^{j+N-3}e^{-c\omega^2} \left|\sin(\sqrt{t}\omega)\right|^2d\omega\\
&\leq &t^{-\frac{j+N-2}{2}}\int_{0}^\infty \omega^{j+N-3}e^{-c\omega^2} d\omega.
\end{eqnarray*}
Now, for  $j+N-3\geq 0$ (which holds for all $j\geq 0$ and $N\geq 3$ or for all $j+N\geq 3$ and $N\geq 1$), we have
\begin{equation*}
\int_{0}^\infty \omega^{j+N-3}e^{-cz^2} dz=\frac{\Gamma \left(\frac{j+N-2}{2}\right)}{2c^{\frac{j+N-2}{2}}}.
\end{equation*}
Consequently, we have from above that
\begin{equation*}
\int_{\R^N}|\xi|^je^{-c|\xi|^2t}\left\vert \frac{\sin (t|\xi|)}{\xi}\right\vert^2d\xi\leq Ct^{-(N-2)/2-j/2},\qquad \text{if}\qquad N+j\geq 3
\end{equation*}
with $C=\frac{1}{2}c^{-\frac{j+N-2}{2}}\Gamma \left(\frac{j+N-2}{2}\right)$,
which is exactly \eqref{Estimate_Sin_N_3}.
\end{proof}

We can now proceed to give and prove our first main result, which reads as follows.

\begin{theorem}\label{Theorem_Main}
Let $u$ be the solution of \eqref{Bose_Problem_Linear_Fourier}-\eqref{Initial_data_Fourier}.  Assume that $0<\tau <\beta$. Let  $V=(\tau u_{tt}+u_t , \partial_x(\tau u_t+ u),\partial_x u_t)$ and assume that $V_0\in L^1(\R^N)\cap H^s(\R^N)$, $s\geq 1$. Then, for all $0\leq j\leq s$, we have
\begin{equation}\label{Estimate_Main_Theorem}
\Vert \partial_x^j V(t)\Vert_{L^2(\R^N)}\leq C(1+t)^{-N/4-j/2}\Vert V_0\Vert_{L^1(\R^N)}+e^{-ct}\Vert \partial_x^j V_0\Vert_{L^2(\R^N)}.
\end{equation}
\end{theorem}
\begin{proof}
First, observe that the Fourier image of $V$ satisfies the decay estimate of Proposition \ref{Main_estimate}. Now, to show \eqref{Estimate_Main_Theorem}, we have from   \eqref{rho_function}   that
\begin{equation}\label{rho_2_behavior_new}
\rho(\xi)\geq\left\{
\begin{array}{ll}
    c|\xi|^2, & \text{if }| \xi|\leq 1, \vspace{0.2cm} \\
   c,  &   \text{if } |\xi|\geq 1.
\end{array}
\right.
\end{equation}
Applying the Plancherel theorem and using the  estimate in  (\ref%
{Main_Estimate_V_Fourier}), we obtain%
\begin{eqnarray}
\left\Vert \partial _{x}^{j}V\left( t\right) \right\Vert _{L^{2}}^{2}
&=&\int_{\mathbb{R}^N }\vert \xi \vert ^{2j}\vert \hat{V}%
\left( \xi ,t\right) \vert ^{2}d\xi  \notag \\
&\leq &C\int_{\mathbb{R}^N }\left\vert \xi \right\vert ^{2j}e^{-c\rho \left(
\xi \right) t}\vert \hat{V}\left( \xi ,0\right) \vert ^{2}d\xi
\notag \\
&=&C\int_{\left\vert \xi \right\vert \leq 1}\left\vert \xi \right\vert
^{2j}e^{-c\rho\left( \xi \right) t}\vert \hat{V}\left( \xi ,0\right)
\vert ^{2}d\xi +C\int_{\left\vert \xi \right\vert \geq 1}\left\vert
\xi \right\vert ^{2j}e^{-c\rho \left( \xi \right) t}\vert \hat{V}\left(
\xi ,0\right) \vert ^{2}d\xi  \notag \\
&=:&I_{1}+I_{2}.  \label{deron_U_equality}
\end{eqnarray}%
Exploiting (\ref{rho_2_behavior_new}),
 we infer that
\begin{equation}
I_{1}\leq C\Vert \hat{V}_{0}\Vert _{L^\infty }^{2}\int_{\left\vert
\xi \right\vert \leq 1}\left\vert \xi \right\vert ^{2j}e^{-c|\xi|
^{2}t}d\xi \leq C\left( 1+t\right) ^{-N/2-j
}\left\Vert V_{0}\right\Vert _{L^{1}}^{2},  \label{I_1_estimate}
\end{equation}

\noindent where we have used the inequality \eqref{Estimate_Integral}.
In the high-frequency region ($|\xi|\geq 1$), we have
\begin{equation*}
I_2 \leq e^{-ct}\int_{\left\vert \xi \right\vert \geq 1}\left\vert \xi
\right\vert ^{2j}\vert \hat{V}\left( \xi ,0\right)
\vert ^{2}d\xi
\leq e^{-ct} \Vert \partial_x^jV_0\Vert^2_{L^2}.
\end{equation*}
Collecting the above two estimates, we obtain \eqref{Estimate_Main_Theorem}.
This finishes the proof of Theorem \ref{Theorem_Main}.
\end{proof}

\begin{remark}
The estimate \eqref{Estimate_Main_Theorem} does not give the decay rate of the solution $u$. In fact, it gives the decay rates of the  norms $\Vert \tau u_{tt}+u_t\Vert_{L^2}$, $\Vert \partial_x(\tau u_t+u)\Vert_{L^2}$ and $\Vert \partial_xu_t\Vert_{L^2}$ (and also of the corresponding derivatives). These three norms are expected to decay with different rates. In order to know the decay rate of the solution and all its spatial derivatives, we need to use the explicit form of the solution and the eigenvalues expansion. This will be done in the following section.
\end{remark}


\section{Eigenvalues expansion}\label{sec:eigenvalues}
In this section, we use the eigenvalues expansion and the explicit form of the Fourier image of the solution in order to find the decay rates of the solution and its spacial derivatives.

The characteritic equation associated to \eqref{Matrix_Form} is
\begin{equation}\label{Characteristic_Polynomial_xi}
\det(L+|\xi|^2 A-\lambda I) = \tau \lambda^3+\lambda^2+\beta |\xi|^2 \lambda+|\xi|^2 = 0.
\end{equation}

The solutions $\lambda_i,\,i=1,2,3$ of the previous equation are the eigenvalues of $\Phi(\xi)$. We will use either $\lambda_i(\xi)$ or $\lambda_i(|\xi|)$ to denote them (depending on which of both notations is more convenient and when no confusion is possible) during the text below.

\vspace{0.2cm}

The following proposition on the description of these eigenvalues is an adaptation of some of the results of Proposition 4 of \cite{P-SM-2015} (some part also in \cite{Trigg_et_al}), that we summarize and adapt here for a better comprehension and to be used later in the present work.

\begin{proposition} [Description of the eigenvalues, \cite{P-SM-2015} and \cite{Trigg_et_al}]\label{descriptioneig}
For each $\xi\in\mathbb{R}^N$ there exist three corresponding eigenvalues of $\Phi(\xi)$, that we name $\lambda_j(|\xi|)$, $j=1,2,3$, the three solutions of the corresponding characteristic equation \eqref{Characteristic_Polynomial_xi}.

We define the following numbers $m_1,m_2$, which are the zeroes of the Cardano discriminant associated to the characteristic equation \eqref{Characteristic_Polynomial_xi}:
          \begin{equation}\label{m1m2}
          m_1 =\tau \dfrac{-C_1-\sqrt{C_2}}{8\beta^3},\qquad  m_2 =\tau \dfrac{-C_1+\sqrt{C_2}}{8\beta^3}
          \end{equation}
            with
          \begin{equation}\label{eq:AB}
              C_1 = 27-18\left(\dfrac{\beta}{\tau }\right) -\left(\dfrac{\beta}{\tau }\right)^2  ,\ \qquad
              C_2 =  C_1^2-64\left(\dfrac{\beta}{\tau }\right)^3.
          \end{equation}
  Under the dissipativeness condition $0<\tau <\beta$, the eigenvalues of $\Phi(\xi)$ satisfy the following:

  \begin{enumerate}[1.]
    \item

     \begin{enumerate}[a)]
       \item $\lambda_{1}(|\xi|)=-\frac{1}{\tau } \textrm{ and }\lambda_{2,3}(|\xi|)=0 \textrm{ when }|\xi|=0$.
       \vspace{0.2cm}
       \item If $\frac{1}{9} <\frac{\tau }{\beta}<1$, $\lambda_1(|\xi|)\in\mathbb{R}\textrm{ and }\lambda_{2}(|\xi|)=\overline{\lambda_{3}(|\xi|)}\in\mathbb{C}\setminus\mathbb{R} \textrm{ for all values of }|\xi|>0$. \vspace{0.2cm}
       \item If $0 <\frac{\tau }{\beta}<\frac{1}{9}$, the type of eigenvalue depends on the value of $|\xi|$ (see Figure \ref{fig:spectrum}). More concretely:
        \vspace{0.2cm}
       \begin{enumerate}
         \item $\lambda_3(|\xi|)\in\mathbb{R}$ and $\lambda_{1}(|\xi|)=\overline{\lambda_{2}(|\xi|)}\in\mathbb{C}\setminus\mathbb{R}$ for $0< |\xi|< \sqrt{m_1}$. \vspace{0.2cm}
         \item $\lambda_{1,2,3}(|\xi|)\in\mathbb{R}\textrm{ for }\sqrt{m_1}\leq |\xi|\leq \sqrt{m_2}.$
         Moreover, in the case that $|\xi|=\sqrt{m_1}$ or $|\xi|=\sqrt{m_2}$, two of these real roots are equal.
        \item $\lambda_1(|\xi|)\in\mathbb{R}$ and $\lambda_{2}(\xi)=\overline{\lambda_{3}(\xi)}\in\mathbb{C}\setminus\mathbb{R}$ for $|\xi|>\sqrt{m_2}$. \vspace{0.2cm}
        \end{enumerate}
        \vspace{0.2cm}
       \item If $\frac{\tau }{\beta}=\frac{1}{9}$, we have $m_1=m_2$ and
         \begin{enumerate}
         \item $\lambda_3(|\xi|)\in\mathbb{R}$ and $\lambda_{1}(|\xi|)=\overline{\lambda_{2}(|\xi|)}\in\mathbb{C}\setminus\mathbb{R}$ for $0< |\xi|< \sqrt{m_1}$. \vspace{0.2cm}
         \item $\lambda_{1,2,3}(|\xi|)= -\frac{3}{\beta}\in\mathbb{R}\quad \textrm{    for } \quad|\xi|=\sqrt{m_1}=\sqrt{m_2}$ (triple root case).
         \item $\lambda_1(|\xi|)\in\mathbb{R}$ and $\lambda_{2}(|\xi|)=\overline{\lambda_{3}(|\xi|)}\in\mathbb{C}\setminus\mathbb{R}$ for $|\xi|>\sqrt{m_2}$. \vspace{0.2cm}
        \end{enumerate}
        \vspace{0.2cm}
      \end{enumerate}

    \item

    \begin{enumerate}[a)]
        \item If $\lambda(|\xi|)$, $|\xi|\neq 0$, is a real eigenvalue of $\Phi(\xi)$, then
        \begin{equation}\label{eq:fita_lambda}
                -\dfrac{1}{\tau } < \lambda(|\xi|) <-\dfrac{1}{\beta}.
              \end{equation}
        If $\lambda(|\xi|)$ is nonreal, then
        \begin{equation}\label{eq:fita_part re}
                \limCo<\Re(\lambda(|\xi|))<0.
              \end{equation}

      \item If $|\xi_1|<|\xi_2|$ and such that $\lambda_{2}(|\xi_1|),\lambda_{2}(|\xi_2|)\in\mathbb{C}\setminus\mathbb{R}$, then $\Re(\lambda_2(|\xi_1|))>\Re(\lambda_2(|\xi_2|))$.

%
%

  \end{enumerate}

\end{enumerate}
\end{proposition}

\begin{figure}[htpb]
\begin{subfigure}[b]{0.5\textwidth}
  \includegraphics[width=\textwidth]{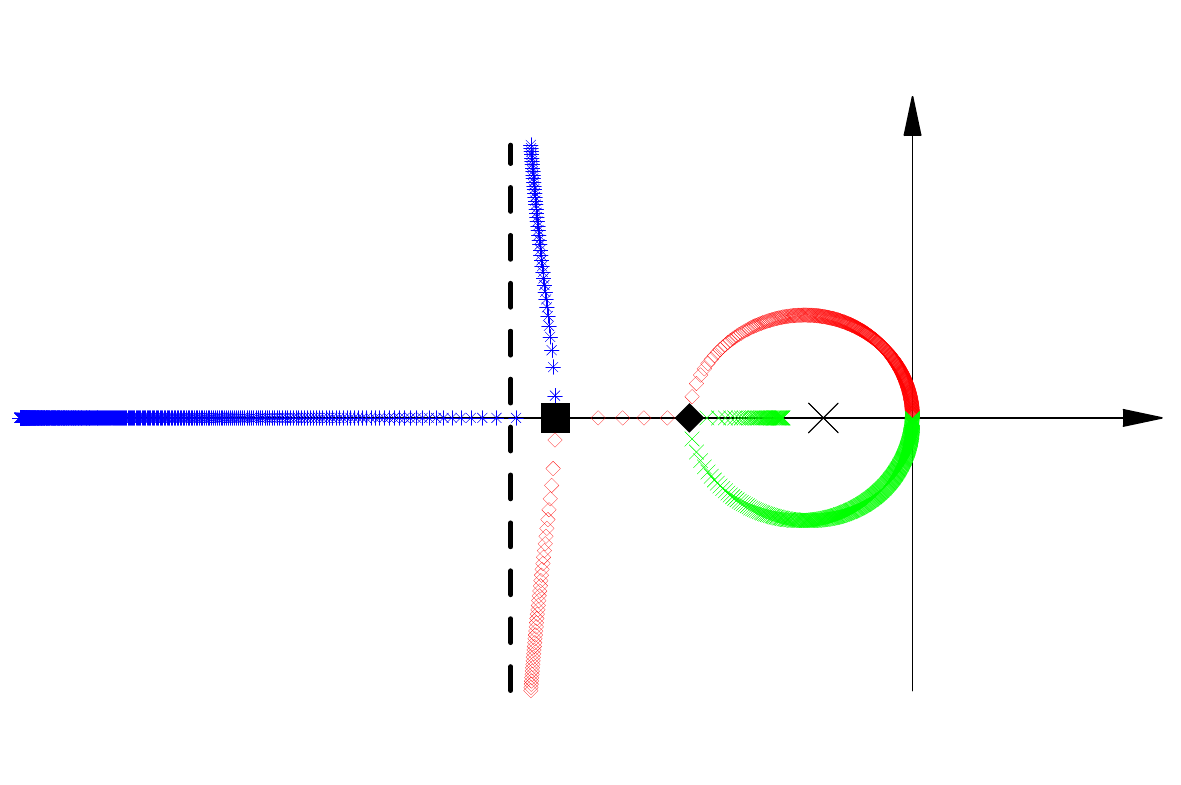}
\end{subfigure}\\ \vspace{-0.7cm}
\begin{subfigure}[b]{1\textwidth}
  \includegraphics[width=\textwidth]{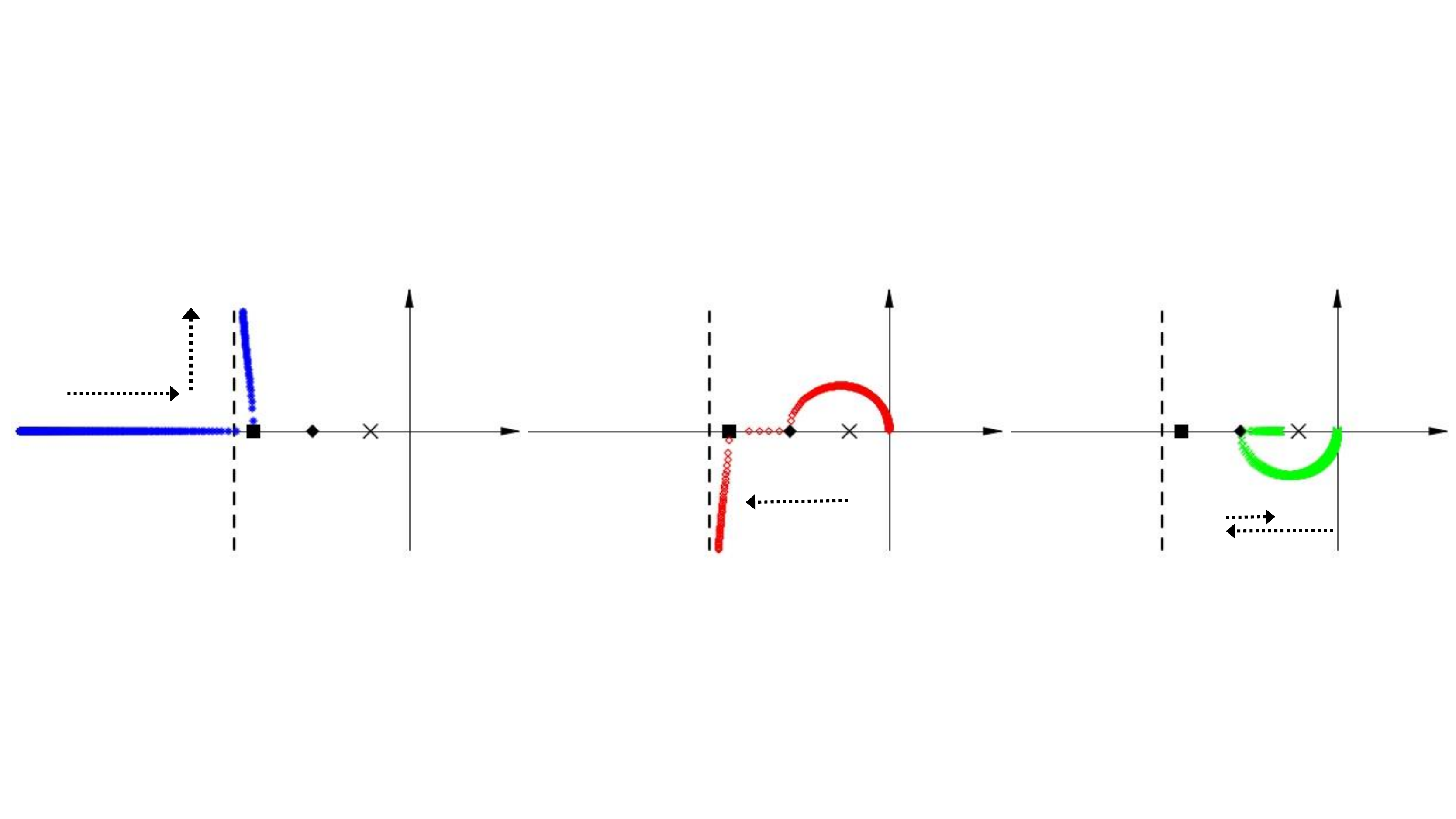}\vspace{-2cm}
\end{subfigure}
\caption{Plot of $\lambda(|\xi|)$  when $0<\frac{\tau}{\beta}<\frac{1}{9}$ (case $(1c)$ of Proposition \ref{descriptioneig}). Top: the three sequences of solutions of the characteristic equation. Bottom (from left to right): separated plot of each of these sequences $\lambda_1(|\xi|)$, $\lambda_2(|\xi|)$ and $\lambda_3(|\xi|)$. The large dot and diamond correspond to the double eigenvalues
$\lambda_1(\sqrt{m_2})=\lambda_2(\sqrt{m_2})$ and $\lambda_2(\sqrt{m_1})=\lambda_3(\sqrt{m_1})$ respectively. The dashed arrows represent how each family of solutions is increasing as a function of $|\xi|$. The dashed vertical line is $\Re(\lambda)=-\frac{1}{2}\left(\frac{1}{\tau}-\frac{1}{\beta}\right)$, which is the limit of the nonreal sequences as $|\xi|\to\infty$, and the cross represents the point $(-\frac{1}{\beta},0)$, which is the limit of the real one as $|\xi|\to\infty$ (see \cite{P-SM-2015} and \cite{Trigg_et_al} for more details).}\label{fig:spectrum}
\end{figure}

\begin{remark}\label{rem:labeleigenv}
  Observe that with the previous labelling of the eigenvalues in part 1 of Proposition \ref{descriptioneig}, each $\lambda_i(|\xi|)$, $i=1,2,3$, is a continuous function in $\xi$. However, during the rest of the paper and for simplicity in the notation we may call $\lambda_1$ the real eigenvalue and $\lambda_{2,3}$ the complex conjugate ones for all $\xi$ in some proofs (it will be mentioned when this is done).
\end{remark}

In the description of the eigenvalues, we also need to prove the following lemma.
\begin{lemma}\label{Lemma_Reiss}
The characteristic equation associated to \eqref{Bose_Problem_Linear_Fourier} has no pure imaginary solution if $0<\tau <\beta$ (dissipative case).
\end{lemma}

\begin{proof}
We consider \eqref{Characteristic_Polynomial_xi}, the characteristic equation associated to \eqref{Bose_Problem_Linear_Fourier}.
Assume that there exists an eigenvalue $\lambda_0(|\xi|)=i\alpha$ as a solution of \eqref{Characteristic_Polynomial_xi} with $\alpha\in \R$. Plugging $\lambda_0$ into \eqref{Characteristic_Polynomial_xi} and  splitting the real and imaginary parts, we obtain
\begin{equation*}
-\tau \alpha^3+\beta |\xi|^2\alpha=0\qquad \text{and}\qquad |\xi|^2-\alpha^2=0.
\end{equation*}

From the first equation, we have two possibilites: $\alpha=0$ or $\alpha=\pm \sqrt{\frac{\beta}{\tau }} |\xi|$. In the first case and using now the second equation, the only possibility is that $\alpha=|\xi|=0$, which actually means that $\lambda_0=0$ is a (double) real eigenvalue when $|\xi|=0$. The second one would be fulfilled only if $|\xi|=0$ or if $\beta=\tau $. If $|\xi|=0$, we would again obtain $\lambda_0=0$ as a (double) real eigenvalue. The case $\beta=\tau $ will not be considered since we assumed that $0<\tau <\beta$. Hence, the characteristic equation associated to \eqref{Bose_Problem_Linear_Fourier} has no pure imaginary solutions in the dissipative case.

%


 \end{proof}

In order to give the decay rate of the solution in the next section, we now proceed to give asymptotic approximations of the eigenvalues of $\Phi(\xi)$ when $|\xi|\to 0$ and $|\xi|\to\infty$. For this purpose, it will be more convenient to apply the change of variables $\zeta=i|\xi|$ in the  characteristic equation \eqref{Characteristic_Polynomial_xi}, that now becomes:
\begin{equation}\label{Characteristic_Polynomial}
\det(L-\zeta^2 A-\lambda I) = \tau \lambda^3+\lambda^2-\beta \zeta^2 \lambda-\zeta^2 = 0.
\end{equation}

Recall that $\lambda_j(\zeta),\, j=1,2, 3$, are the roots of \eqref{Characteristic_Polynomial}
, that we write
\begin{equation}\label{Asymptotic_Expansion}
\lambda_j(\zeta)=\lambda_j^{(0)}+\lambda_j^{(1)}\zeta+\lambda_j^{(2)}\zeta^2+...,\qquad \qquad j=1,2, 3.
\end{equation}
or, equivalently,
\begin{equation*}
\lambda_j(|\xi|)=\lambda_j^{(0)}+\lambda_j^{(1)}\, i\, |\xi|\,-\lambda_j^{(2)}|\xi|^2+...,\qquad \qquad j=1,2, 3.
\end{equation*}

For simplicity, we will denote now as $\lambda_1$ the real root and $\lambda_{2,3}$ the complex conjugate ones when $|\xi|\rightarrow 0$, both when $0<\tau /\beta<1/9$ and $1/9<\tau /\beta<1$ (see Remark \ref{rem:labeleigenv} and Proposition \ref{descriptioneig}).

We can now compute the first coefficients in \eqref{Asymptotic_Expansion} using the characteristic equation \eqref{Characteristic_Polynomial}, obtaining that
\begin{equation}\label{Eigenvalues_L}
\left\{
\begin{array}{ll}
\lambda _1^{(0)}=-\dfrac{1}{\tau }
\vspace{0.3cm}\\
\lambda_j^{(0)}=0,\quad \lambda_j^{(1)}=\pm 1, \quad \lambda_j^{(2)}=\dfrac{1}{2}(\beta-\tau )&\qquad \text{for }\quad  j=2,3.
\end{array}
\right.
\end{equation}

Consequently, we have for $|\xi|\rightarrow 0$ that

\begin{equation}\label{Estimate_Eigenvalues_0}
\func{Re}(\lambda_j(|\xi|))=
\left\{
\begin{array}{ll}
-\dfrac{1}{\tau }+O(|\xi|),&\qquad \text{for}\qquad j=1,
\vspace{0.2cm}\\
-\dfrac{1}{2}(\beta-\tau )|\xi|^2+O(|\xi|^3),&\qquad \text{for}\qquad j=2,3.
\end{array}
\right.
\end{equation}

Under the assumption $0<\tau <\beta$, it is clear that $\func{Re}(\lambda_j)<0$ for all $j=1,2,3$ when $|\xi|\rightarrow 0$.

\begin{remark}\label{Remark_Behavior_eigenvalues}
The behavior of the solution of \eqref{Matrix_Form} depends on the behavior of the function $e^{\Re{\lambda_j}(\xi)t}$, $j=1,2,3$. Since in most cases $\lambda_j(|\xi|)$ is a power series of $|\xi|$, so it is its real part. Observe that as $\Re{\lambda_j}(|\xi|)<0$, the frequencies that give the dominant part of all $e^{\Re{\lambda_j}(|\xi|)t}$ are those corresponding to small frequencies $|\xi|$. For this reason, the behavior of the real part near $|\xi|=0$ determines  the decay rate of the solution. For large frequencies, and again as $\Re{\lambda_j}(|\xi|)<0$, it is clear that $e^{\Re{\lambda_j}(|\xi|)t}$ can be always estimated by $e^{-ct}$ if the powers in the Taylor series expansion of $\Re{\lambda_j}(|\xi|)$ near infinity are positive or by $|\xi|^me^{-c|\xi|^{-\hat{c}}t}$ for a certain $m>0$ if one of the powers in the Taylor series expansion is negative. In both cases and using Plancherel's theorem, we see that the integral in the high frequencies is bounded if and only if some derivatives of the solution are bounded, which means that the solution should be in some Sobolev spaces and this gives the regularity of initial data needed for the desired decay rate.
\end{remark}

\vspace{0.3cm}
Next, we proceed to give asymptotic approximations of the eigenvalues when $|\zeta|\rightarrow \infty$. Following \cite{IHK08},
we can take $\eta=\zeta^{-1}=(i|\xi|)^{-1}$ and write the equation \eqref{Matrix_Form}, with $L$ and $A$ defined in \eqref{LA}, as
$$ \hat{U}_t(\eta,t)=\eta^{-2}\left( L \eta^2-A\right) \hat{U}(\eta,t),$$
whose characteristic equation writes as
\begin{equation}\label{Charact_Equation_mu}
\det(L \eta^2-A-\mu I) = \tau \mu^3+\eta^{2}\mu^2-\beta  \eta^{2} \mu-\eta^{4}=0.
\end{equation}
Observe that we have the relation $$\lambda_j(\zeta)=\zeta^2\mu_j(\zeta^{-1})$$ between $\mu_j(\eta)$ and $\lambda_j(\zeta)$, solutions of \eqref{Characteristic_Polynomial} and \eqref{Charact_Equation_mu} respectively.

Now, for $\eta\rightarrow 0$ 
we can write $\mu_j(
\eta)$ as:
\begin{equation}\label{Asymptotic_Expansion_mu}
\mu_j(\eta)=\mu_j^{(0)}+\mu_j^{(1)}\eta+\mu_j^{(2)}\eta^{2}+...,\qquad \qquad j=1,2, 3.
\end{equation}
Or, equivalently,
\begin{equation*}
\lambda_j(|\xi|)=-\mu_j^{(0)}|\xi|^2+\mu_j^{(1)} i |\xi| +\mu_j^{(2)}+...,\qquad \qquad j=1,2, 3.
\end{equation*}
Plugging \eqref{Asymptotic_Expansion_mu} into \eqref{Charact_Equation_mu}, we get, after performing some computations,
\begin{equation*}
\left\{
\begin{array}{ll}
\mu_j^{(0)}=0,&\qquad \text{for}\qquad j=1,2,3\vspace{0.2cm}\\
\mu_j^{(1)}=0,\qquad\qquad \qquad  \mu_j^{(2)}=-\dfrac{1}{\beta}&\qquad \text{for}\qquad j=1,\vspace{0.2cm}\\
\mu_j^{(1)}=\pm \sqrt{\dfrac{\beta}{\tau }}, \qquad\qquad  \mu_j^{(2)}=-\dfrac{(\beta-\tau )}{2\beta\tau } &\qquad \text{for}\qquad j=2,3.
\end{array}
\right.
\end{equation*}
Consequently, we deduce from above that for $|\xi|\rightarrow \infty$ we have

\begin{equation}\label{Estimate_Eigenvalues_infinity}
\func{Re}(\lambda_j(|\xi|))=
\left\{
\begin{array}{ll}
-\dfrac{1}{\beta}+O(|\xi|^{-1}),&\qquad \text{for}\qquad j=1,\vspace{0.2cm}\\
-\dfrac{(\beta-\tau )}{2\beta\tau }+O(|\xi|^{-1}),&\qquad \text{for}\qquad j=2,3,
\end{array}
\right.
\end{equation}
where we are denoting as $\lambda_1$ the real root and $\lambda_{2,3}$ the complex conjugate ones when $|\xi|\rightarrow \infty$ (see Proposition \ref{descriptioneig}). Under the assumption $0<\tau <\beta$, it is clear that $\func{Re}(\lambda_j)<0$ for all $j=1,2,3$ when $|\xi|\rightarrow \infty$.

Let us now divide the frequency space into three regions: low frequency, high frequency and middle frequency region, that is

\begin{equation*}
\Upsilon_{L}=\left\{\xi\in\R^N; |\xi|<\nu_1\ll 1\right\},\,\Upsilon_{H}=\left\{\xi\in\R^N; |\xi|>\nu_2\gg 1\right\},\,\Upsilon_{M}=\left\{\xi\in\R^N; \nu_1\leq|\xi|\leq \nu_2\right\}.
\end{equation*}

The choice of $\nu_1$ and $\nu_2$ will be discussed in the proofs of Proposition \ref{Proposition_Low} and of Lemma \ref{Lemma_Y_M_2}. For the moment, we need $\nu_1$ and $\nu_2$ sufficiently small and large, respectively, such that the asymptotic expansions of Propositions \ref{Proposition_Low} and \ref{Proposition_high} hold.

We write the solution of the system \eqref{Bose_Problem_Linear_Fourier} in the above  three regions. In the following Propositions \ref{Proposition_Low}, \ref{Proposition_high} and \ref{Proposition_Middle} we give bounds of the solution on each of this three regions using the previous asymptotic expansions of the eigenvalues. These bounds will be used in Theorem \ref{Theorem_Solution_Energy_Space} to proof the decay estimate of the solution of problem \eqref{Bose_Problem}.

\begin{proposition}\label{Proposition_Low}
If $0<\tau<\beta$, the solution $\hat{U}(\xi,t)$ of \eqref{Matrix_Form} satisfies, for all $\xi\in\Upsilon_{L}$ with $|\xi|\neq 0$, the estimates:


\begin{eqnarray}\label{Low_F_Estimate_1}
|\hat{u}(\xi,t)|&\leq& C_L\left(|\xi|^2|\hat{u}_0|+|\xi|^2|\hat{u}_1|+|\hat{u}_2|\right) e^{-c_1t}  \notag\\
                    && +C_L\left(|\hat{u}_0|+|\xi|^2|\hat{u}_1|+|\hat{u}_2|\right) e^{-c_2|\xi|^2 t}\cos(|\xi| t)  \notag\\
                    &&+ C_L \left( |\xi||\hat{u_0}|+\dfrac{1}{|\xi|}|\hat{u}_1|+\dfrac{1}{|\xi|}|\hat{u}_2|\right) e^{-c_2|\xi|^2 t}\sin(|\xi| t),\qquad \textrm{for all } t\geq 0.
\end{eqnarray}
with $c_1=\frac{1}{\tau }$ and $c_2=\frac{\beta-\tau }{2}$ and $C_L=C_L(\beta,\tau )>0$ (all positive constants). Moreover, if $\int_{\R^N}u_1(x)dx=\int_{\R^N}u_2(x)dx=0$ we have

\begin{eqnarray}\label{Low_F_Estimate}
|\hat{u}(\xi,t)|&\leq& C_L\left(|\xi|^2|\hat{u}_0|+|\xi|^2|\hat{u}_1|+|\hat{u}_2|\right) e^{-c_1t}  \notag\\
                    && +C_L\left(|\hat{u}_0|+|\xi|^2|\hat{u}_1|+|\hat{u}_2|\right) e^{-c_2|\xi|^2 t}\cos(|\xi| t)\notag\\
 &&+C_L\left( |\xi||\hat{u_0}|+\Vert u_1\Vert_{L^{1,1}}+\Vert u_2\Vert_{L^{1,1}}\right)e^{-c_2|\xi|^2 t}\sin (|\xi| t)
,\qquad \textrm{for all }  t\geq 0.
\end{eqnarray}
where $L^{1,1}$ is the $L^1$-weighted space defined by
\begin{equation}\label{eq:L11new}
L^{1,1}(\R^N)=\left\{u\in L^{1}\left( \mathbb{R}^N\right) ; \Vert u_1\Vert_{L^{1,1}(\R^N)}=\int_{\R^N} (1+|x|)|u(x)|dx<\infty \right\}.
\end{equation}

\end{proposition}

\begin{remark}\label{rem:measurezero1}
  Observe that the estimate \eqref{Low_F_Estimate_1} is not satisfied if $|\xi|=0$ but, as it is a set of measure zero, it will not affect the decay of the solution in Theorem \ref{Theorem_Solution_Energy_Space}.
\end{remark}
\begin{proof}

According to part 1 of Proposition \ref{descriptioneig}, if $\xi\in\Upsilon_L$ is such that $|\xi|$ is small enough (that is, $|\xi|<\nu_1<\sqrt{m_1}$) we know that there exist one real root and two complex conjugate ones. For simplicity, we are going to denote $\lambda_1$ the real root and $\lambda_{2,3}$ the complex conjugate ones when $|\xi|\rightarrow 0$, both when $0<\tau /\beta<1/9$ and $1/9<\tau /\beta<1$ (see Remark \ref{rem:labeleigenv} and Proposition \ref{descriptioneig}). Hence, the solution of the equation \eqref{Bose_Problem_Linear_Fourier} when $\xi\in\Upsilon_L$ can be written in terms of the corresponding eigenvalues as

\begin{equation}\label{Solution_Formula_L}
\hat{u}(\xi,t)=C_1(\xi)e^{\lambda_1(\xi)t}+e^{ \func{Re}(\lambda_2(\xi))t}\left[C_2(\xi)\cos (\func{Im}(\lambda_2(\xi))t)+C_3(\xi)\sin (\func{Im}(\lambda_2(\xi))t)\right].
\end{equation}
We may use the initial values \eqref{Initial_data_Fourier} to find the above constants by solving the system
\begin{equation}\label{C_system}
\left\{
\begin{array}{ll}
C_1+C_2=\hat{u}_0,\vspace{0.2cm } \\
\lambda_1 C_1 +  \func{Re}(\lambda_2) C_2 + \func{Im}(\lambda_2) C_3 = \hat{u}_1,\vspace{0.2cm } \\
(\lambda_1)^2 C_1 +  \left( (\func{Re}(\lambda_2))^2-(\func{Im}(\lambda_2))^2\right) C_2 + 2 \func{Re}(\lambda_2)\func{Im}(\lambda_2) C_3 = \hat{u}_2
\end{array}
\right.
\end{equation}
(we are omitting the $\xi$ dependence in order to simplify the notation). By neglecting  the small terms, we have from \eqref{Estimate_Eigenvalues_0} that in $\Upsilon_L$ the eigenvalues are:
\begin{equation}\label{asympt_eigenv_xi0}
\left\{
\begin{array}{ll}
\lambda_1(\xi) \sim  -\dfrac{1}{\tau },&\vspace{0.3cm}\\
\lambda_2(\xi) \sim i|\xi|-\dfrac{1}{2}(\beta-\tau )|\xi|^2,&\vspace{0.3cm}\\
\lambda_3(\xi)\sim -i|\xi|-\dfrac{1}{2}(\beta-\tau )|\xi|^2. &
\end{array}
\right.
\end{equation}
Solving the system \eqref{C_system} and using the asymptotic expressions of the eigenvalues when $|\xi|\to 0$ in \eqref{asympt_eigenv_xi0} we get that, in $\Upsilon_L$,
\begin{equation*}
\left\{
\begin{array}{ll}
C_1(\xi)= \left( -\tau ^2|\xi|^2 + O(|\xi|^3)\right) \hat{u}_0+ \left( -\tau ^2(\beta-\tau )|\xi|^2 + O(|\xi|^3)\right) \hat{u}_1+ \left( -\tau ^2 + O(|\xi|)\right) \hat{u}_2,\vspace{0.3cm }\\
C_2(\xi)= \left( -1 + O(|\xi|)\right) \hat{u}_0+ \left( \tau ^2(\beta-\tau )|\xi|^2 + O(|\xi|^3)\right) \hat{u}_1 + \left( \tau ^2 + O(|\xi|)\right) \hat{u}_2, \vspace{0.3cm }\\
C_3(\xi)= \left( -\frac{\beta+\tau }{2}|\xi| + O(|\xi|^2)\right) \hat{u}_0+ \left( -\frac{1}{|\xi|} + O(1)\right) \hat{u}_1+ \left( -\frac{\tau }{|\xi|} + O(1)\right) \hat{u}_2.
\end{array}
\right.
\end{equation*}

We can now take the following approximate solution of \eqref{Bose_Problem_Linear_Fourier} and \eqref{Initial_data_Fourier},
\begin{equation*}
  \tilde{u}(\xi,t) = \widetilde{C}_1(\xi) e^{-\frac{1}{\tau }t} + e^{-\frac{\beta-\tau }{2}|\xi|^2 t}\left( \widetilde{C}_2(\xi) \cos(|\xi|  t) + \widetilde{C}_3(\xi) \sin(|\xi|  t)  \right)
\end{equation*}
where
\begin{equation*}
\left\{
\begin{array}{ll}
\widetilde{C}_1(\xi)=  -\tau ^2|\xi|^2  \hat{u}_0 -\tau ^2(\beta-\tau )|\xi|^2  \hat{u}_1 -\tau ^2  \hat{u}_2,\vspace{0.3cm }\\
\widetilde{C}_2(\xi)=  -\hat{u}_0+ \tau ^2(\beta-\tau )|\xi|^2  \hat{u}_1 + \tau ^2  \hat{u}_2, \vspace{0.3cm }\\
\widetilde{C}_3(\xi)=  -\frac{\beta+\tau }{2}|\xi|  \hat{u}_0 -\frac{1}{|\xi|}  \hat{u}_1 -\frac{\tau }{|\xi|}  \hat{u}_2.
\end{array}
\right.
\end{equation*}

Observe that, at a first leading order, solving the system \eqref{C_system} is equivalent to solving
\begin{equation*}
\left\{
\begin{array}{ll}
C_1+C_2=\hat{u}_0,\vspace{0.2cm }\\
-\dfrac{C_1}{\tau }-\dfrac{1}{2} C_2 |\xi| ^2 (\beta -\tau )+C_3 |\xi|=\hat{u}_1,\vspace{0.2cm}\\
\dfrac{C_1}{\tau ^2}+\left(\dfrac{1}{4}  |\xi| ^4 (\beta -\tau )^2-| \xi |^2\right)C_2-C_3 |\xi| ^3 (\beta -\tau )=\hat{u}_2.
\end{array}
\right.
\end{equation*}
which has the previous $\widetilde{C}_1$, $\widetilde{C}_2$, $\widetilde{C}_3$ as exact solution.

It is immediate to see that $\tilde{u}(\xi,t)$ satisfies the bound given in \eqref{Low_F_Estimate_1}. Also, from the previous calculus on the asymptotic expressions of $C_1,C_2$ and $C_3$ it is clear that
\begin{equation*}
\lim_{|\xi|\to 0}  \left| \hat{u}(\xi,t)-\tilde{u}(\xi,t) \right| = 0\qquad  \textrm{ for all } t\geq 0,
\end{equation*}
where $\hat{u}(\xi,t)$ is the exact solution given in \eqref{Solution_Formula_L}. Hence, the bound \eqref{Low_F_Estimate_1} is also satisfied by the exact solution $\hat{u}(\xi,t)$.

To prove \eqref{Low_F_Estimate}, we assume  that $\int_{\R^N}u_1(x)dx=\int_{\R^N}u_2(x)dx=0$. Then, we may show (see \cite[Lemma 3.1]{Ike04})
\begin{equation}\label{L11bound}
|\hat{u}_i(\xi)|\leq |\xi|\Vert u_i\Vert_{L^{1,1}(\R^N)},\qquad i=1,2,
\end{equation}
where $L^{1,1}(\R^N)$ is the $L^1$ weighted space defined in \eqref{eq:L11new}.
Consequently, for $|\xi|$ small enough, we have
\begin{equation*}
C_3\leq C\left( |\xi||\hat{u_0}|+\Vert u_1\Vert_{L^{1,1}(\R^N)}+\Vert u_2\Vert_{L^{1,1}(\R^N)}\right).
\end{equation*}
Therefore, taking the last inequality into account, the solution in \eqref{Solution_Formula_L}  satisfies \eqref{Low_F_Estimate}.
\end{proof}

\begin{proposition}\label{Proposition_high}
If $0<\tau<\beta$, the solution $\hat{u}(\xi,t)$ of \eqref{Bose_Problem_Linear_Fourier} satisfies in $\Upsilon_{H}$ the estimate:
\begin{equation}\label{High_F_Estimate}
|\hat{u}(\xi,t)|\leq C_H\left(\left(1+\frac{1}{|\xi|}+\frac{1}{|\xi|^2}\right)|\hat{u}_0(\xi)|+\left(\frac{1}{|\xi|}+\frac{1}{|\xi|^2}\right) |\hat{u}_1(\xi)|+\left(\frac{1}{|\xi|^2}+\frac{1}{|\xi|^3}\right)|\hat{u}_2(\xi)|\right)e^{-c_3 t},
\end{equation}
for all $t\geq 0$, where $c_3=\min\left\{ \frac{1}{\beta},\frac{\beta-\tau }{2\beta\tau } \right\}$  and $C_H=C_H(\beta,\tau )>0$ (all positive constants).
\end{proposition}
\begin{proof}
According to part 1 of Proposition \ref{descriptioneig}, if $\xi\in\Upsilon_H$ such that $|\xi|$ is large enough (that is, $|\xi|>\nu_2>\sqrt{m_2}$) we know that there exist one real root and two complex conjugate ones, namely $\lambda_1$ and $\lambda_{2,3}$. So, as before, the solution of \eqref{Bose_Problem_Linear_Fourier} can be written as
\begin{equation}\label{Solution_Formula_H}
\hat{u}(\xi,t)=D_1(\xi)e^{\lambda_1(\xi)t}+e^{ \func{Re}(\lambda_2(\xi))t}\left[D_2(\xi)\cos (\func{Im}(\lambda_2(\xi))t)+D_3(\xi)\sin (\func{Im}(\lambda_2(\xi))t)\right].
\end{equation}
where $D_i(\xi)$ can be written in terms of the initial data \eqref{Initial_data_Fourier} and hence satisfy the same system as in Proposition \ref{Proposition_Low}, \eqref{C_system}. From \eqref{Estimate_Eigenvalues_infinity} and by neglecting the small terms, we also know that, when $|\xi|\to \infty$,
\begin{equation}\label{asympt_eigenv_xi_infty}
\left\{
\begin{array}{ll}
\lambda_1(\xi) \sim -\dfrac{1}{\beta},&\vspace{0.3cm}\\
\lambda_2(\xi)\sim -\dfrac{\beta-\tau }{2\beta\tau }+i|\xi|\sqrt{\dfrac{\beta}{\tau }},&\qquad \vspace{0.3cm}\\
\lambda_3(\xi) \sim  -\dfrac{\beta-\tau }{2\beta\tau }-i|\xi|\sqrt{\dfrac{\beta}{\tau }}.&
\end{array}
\right.
\end{equation}

Hence, observe that at a first leading order, solving the corresponding system is equivalent to solving
\begin{equation*}
\left\{
\begin{array}{ll}
D_1+D_2=\hat{u}_0,\vspace{0.2cm }\\
-\dfrac{D_1}{\beta }+\dfrac{D_2 (\tau -\beta )}{2 \beta \tau }+D_3 |\xi|  \sqrt{\dfrac{\beta }{\tau }}=\hat{u}_1,\vspace{0.2cm}\\
\dfrac{D_1}{\beta ^2}+\dfrac{D_2 \left( (\beta -\tau )^2-4  \beta^3 \tau  |\xi| ^2\right)}{4 \beta^2 \tau ^2  }+\dfrac{D_3 |\xi|  \sqrt{\dfrac{\beta }{\tau }} (\tau -\beta )}{\beta \tau }=\hat{u}_2.
\end{array}
\right.
\end{equation*}

Solving the corresponding system and using the asymptotic expressions of the eigenvalues when $|\xi|\to\infty$ in \eqref{asympt_eigenv_xi_infty} we get that in $\Upsilon_H$
\begin{equation*}
\left\{
\begin{array}{ll}
{D}_1(\xi)=  \left( 1+   O(|\xi|^{-1}) \right) \hat{u}_0  +\left(\frac{\beta-\tau }{\beta^2|\xi|^2} + O(|\xi|^{-3}) \right) \hat{u}_1 + \left(\frac{\tau }{\beta|\xi|^2} + O(|\xi|^{-3}) \right) \hat{u}_2,\vspace{0.3cm }\\
{D}_2(\xi)=  \left(\frac{2\tau -\beta}{\beta^3|\xi|^2} + O(|\xi|^{-3}) \right)\hat{u}_0 + \left(\frac{\tau -\beta}{\beta^2|\xi|^2} +O(|\xi|^{-3}) \right) \hat{u}_1 - \left(\frac{\tau }{\beta|\xi|^2} + O(|\xi|^{-3}) \right)\hat{u}_2, \vspace{0.3cm }\\
{D}_3(\xi)=  \left(\frac{1}{\beta\sqrt{\beta/\tau }|\xi|} + O(|\xi|^{-2}) \right) \hat{u}_0 + \left(\frac{1}{\sqrt{\beta/\tau }|\xi|} + O(|\xi|^{-2}) \right) \hat{u}_1 +\left( \frac{3\tau -\beta}{2\beta^2\sqrt{\beta/\tau }|\xi|^3} + O(|\xi|^{-4}) \right) \hat{u}_2.
\end{array}
\right.
\end{equation*}

In a similar way as in Proposition \ref{Proposition_Low}, we deduce that \eqref{High_F_Estimate} holds.

\end{proof}

\begin{lemma}\label{Lemma_Y_M_2}
There is a constant $c_4>0$, such that, for all $\xi\in\Upsilon_M$,
\begin{equation}\label{Bound_M_Eigenvalues}
\func{Re}(\lambda_j(\xi))<-c_4<0,
\end{equation}
where $\lambda_j(\xi),\, j=1,2,3$ are the eigenvalues of the matrix $\Phi(\xi)$.
\end{lemma}

\begin{proof}
Let us recall that $\Upsilon_{M}=\left\{\xi\in\R^N;\nu_1\leq|\xi|\leq \nu_2\right\}$ and that we have chosen $\nu_1$ and $\nu_2$ sufficiently small and large, respectively, such that the asymptotic expansions of Propositions \ref{Proposition_Low} and \ref{Proposition_high} hold.  That is, we have chosen $\nu_1$ and ${\nu_2}$ such that $0<\nu_1<\sqrt{m_1}$ and $\nu_2>\sqrt{m_2}$, where $m_1,m_2$ are the constants defined in \eqref{m1m2} and \eqref{eq:AB}. Let us call $\xi_{\nu_1}$ and $\xi_{\nu_2}$ those $\xi$ such that $|\xi_{\nu_1}|=\nu_1$ and $|\xi_{\nu_2}|={\nu_2}$, respectively.

First, suppose that $\frac{1}{9}<\frac{\tau }{\beta}<1$. According to parts 1.b), 2.a) and 2.b) of Proposition \ref{descriptioneig}, $\lambda_1(\xi)\in\mathbb{R}$ and $\lambda_{2,3}(\xi)\in\mathbb{C}\setminus\mathbb{R}$ for all $|\xi|>0$ (in particular, for all $\xi\in \Upsilon_M$) and fulfill
\begin{equation}\label{eq:bounds}
  \lambda_1(\xi)<-\frac{1}{\beta}\qquad \textrm{ and }\qquad \Re(\lambda_{2,3}(\xi))<   \Re(\lambda_{2,3}(\xi_{\nu_1}))<0.
\end{equation}


Suppose now that $0<\frac{\tau }{\beta}<\frac{1}{9}$. From the choice of $\nu_1,\nu_2$ (see above) and according to Proposition \ref{descriptioneig}, we have $\lambda_{2}(\xi_{\nu_1}), \lambda_{2}(\xi_{\nu_2})\in\mathbb{C}\setminus\mathbb{R}$ and we can actually divide $\Upsilon_M$ in three parts: $\xi\in\Upsilon_M$ such that $|\xi|\in[\nu_1,\sqrt{m_1}) $, $\xi\in\Upsilon_M$ such that $|\xi|\in [\sqrt{m_1},\sqrt{m_2}] $ and $\xi\in\Upsilon_M$ such that $|\xi| \in(\sqrt{m_2},\nu_2]$. For those $\xi$ such that $|\xi|\in [\nu_1,\sqrt{m_1}) \cup (\sqrt{m_2},\nu_2]$, the same bounds for the real and complex eigenvalues that in the previous case hold. For those $\xi$ such that $|\xi|\in [\sqrt{m_1},\sqrt{m_2}]$, we recall that $\lambda_{1,2,3}(\xi)\in\mathbb{R}$ (see part 2.c) of Proposition \ref{descriptioneig}) and, hence, we can use part 2.a) of this proposition and we obtain $\lambda_{1,2,3}(\xi) < -\frac{1}{\beta}$.


Finally, consider the special case in which $\frac{\tau }{\beta}=\frac{1}{9}$. Again according to Proposition \ref{descriptioneig}, \eqref{eq:bounds} holds for all $\xi\in\Upsilon_M$ with $|\xi|\neq \sqrt{m_1}=\sqrt{m_2}$, which is the one in which the eigenvalue is a triple real one.  Actually, this triple eigenvalue is $\lambda_{1,2,3}=-\frac{3}{\beta}<-\frac{1}{\beta}$ and, hence, \eqref{eq:bounds} holds for all $\xi\in\Upsilon_M$ when $\frac{\tau }{\beta}=\frac{1}{9}$.

Therefore, we can conclude that \eqref{Bound_M_Eigenvalues} holds in $\Upsilon_M$ with $c_4=\min\left\{ \frac{1}{\beta},|\Re(\lambda_{2,3}(\xi_{\nu_1}))| \right\}>0$. 

\end{proof}

\begin{proposition}\label{Proposition_Middle}
There exists two positive constants $C_M$ and ${c}_4$ such that the solution $\hat{u}(\xi,t)$ of \eqref{Bose_Problem_Linear_Fourier} satisfies in $\Upsilon_{M}$ one of the following estimates:
\begin{equation}\label{Middle_F_Estimate}
|\hat{u}(\xi,t)|\leq C_M\left(|\hat{u}_0(\xi)|+ |\hat{u}_1(\xi)|+|\hat{u}_2(\xi)|\right)e^{-c_4 t},\quad  \textrm{ if } |\xi|\neq 0,\sqrt{m_1},\sqrt{m_2}.
\end{equation}
or
\begin{equation}\label{Middle_F_Estimate_2}
|\hat{u}(\xi,t)|\leq C_M(1+t)\left(|\hat{u}_0(\xi)|+ |\hat{u}_1(\xi)|+|\hat{u}_2(\xi)|\right)e^{-c_4 t}, \ \textrm{ if } |\xi|= 0, \textrm{ or }\frac{\tau }{\beta}\neq\frac{1}{9} \textrm{ and }|\xi|=\sqrt{m_1}, \sqrt{m_2},
\end{equation}
or
\begin{equation}\label{Middle_F_Estimate_3}
|\hat{u}(\xi,t)|\leq C_M(1+t+t^2)\left(|\hat{u}_0(\xi)|+ |\hat{u}_1(\xi)|+|\hat{u}_2(\xi)|\right)e^{-\frac{3}{\beta} t},\quad \textrm{ if } \frac{\tau }{\beta}=\frac{1}{9} \textrm{ and } |\xi|= \sqrt{m_1} =\sqrt{m_2},
\end{equation}
for all $t\geq 0,$ where $c_4$ is defined in the proof of Lemma \ref{Lemma_Y_M_2} and $C_M=C_M(\beta,\tau )$.

\end{proposition}

\begin{remark}\label{rem:measurezero2}
 Observe that the estimates \eqref{Middle_F_Estimate_2} and \eqref{Middle_F_Estimate_3} are satisfied in a set of measure zero, so they will not affect the decay of the solution in Theorem \ref{Theorem_Solution_Energy_Space}.
\end{remark}
\begin{proof}
First, and according to Proposition \ref{descriptioneig}, it is clear that for all $\xi\in \Upsilon_M$, the characteristic equation \eqref{Characteristic_Polynomial_xi} has three roots satisfying one of the following cases:
\begin{itemize}
\item one real and two complex conjugate roots (see the cases in Proposition \ref{descriptioneig});
\item three distinct real roots (see the cases in Proposition \ref{descriptioneig});
\item there is one real root and another real root of double multiplicity (if $|\xi|=0$, or $\frac{\tau }{\beta}\neq\frac{1}{9}$ and $|\xi|=\sqrt{m_1} \textrm{ or }\sqrt{m_2}$);
\item  a real root with triple multiplicity (if $\frac{\tau }{\beta}=\frac{1}{9} \textrm{ and } |\xi|= \sqrt{m_1} =\sqrt{m_2}$).
\end{itemize}
In the following, we discuss the above four cases.

First, suppose \eqref{Characteristic_Polynomial_xi} has one  real root and two  complex conjugate ones, that, for simplicity of notation, we will call $\lambda_1(\xi)$ and $\lambda_{2,3}(\xi)$ respectively (see Remark \ref{rem:labeleigenv}).  Then the solution is written as in \eqref{Solution_Formula_L}:
\begin{equation}\label{Solution_Formula_2}
\hat{u}(\xi,t)=C_1(\xi)e^{\lambda_1(\xi)t}+e^{ \func{Re}(\lambda_2(\xi))t}\left[C_2(\xi)\cos (\func{Im}(\lambda_2(\xi))t)+C_3(\xi)\sin (\func{Im}(\lambda_2(\xi))t)\right]
\end{equation}
with $C_1,\, C_2$ and $C_3$ satisfying \eqref{C_system}. It is clear that $C_i(\xi)$ (that is, $C_i(|\xi|)),\,i=1,2,3$, are bounded in the compact set $\Upsilon_M$.
Thus,  there exists a constant $C>0$ depending on the bounding constant and $\nu_1$ and $\nu_2$ such that
\begin{equation}\label{C_i_Estimate}
|C_i(\xi)|\leq C(|\hat{u}_0(\xi)|+|\hat{u}_1(\xi)|+|\hat{u}_2(\xi)|), \qquad \textrm{ for all } \xi \in \Upsilon_M \qquad (i=1,2,3).
\end{equation}
This last inequality together with \eqref{Bound_M_Eigenvalues} and \eqref{Solution_Formula_2} leads to \eqref{Middle_F_Estimate}.

Second, if the roots of \eqref{Characteristic_Polynomial_xi} are real and distinct, then the solution of \eqref{Bose_Problem_Linear_Fourier} is written as
\begin{equation}\label{Solution_Formula_3}
\hat{u}(\xi,t)=C_1(\xi)e^{\lambda_1(\xi)t}+C_2(\xi)e^{\lambda_2(\xi)t}+C_3(\xi)e^{\lambda_3(\xi)t},
\end{equation}
where $C_1,\, C_2$ and $C_3$ are satisfying the system
\begin{equation}\label{C_system_2}
\left\{
\begin{array}{ll}
C_1+C_2+C_3=\hat{u}_0,\vspace{0.2cm } \\
\lambda_1 C_1 +  \lambda_2 C_2 +\lambda_3 C_3 = \hat{u}_1,\vspace{0.2cm } \\
\lambda_1^2 C_1 +  \lambda_2^2 C_2 +\lambda_3 ^2C_3= \hat{u}_2
\end{array}
\right.
\end{equation}
(that this system was obtained imposing that $\hat{u}(\xi,t)$ must satisfy the initial conditions \eqref{Initial_data_Fourier}).
It is not hard to see that \eqref{C_i_Estimate} also holds and therefore (as before) \eqref{Middle_F_Estimate} is also satisfied.

Third, we assume that there exists $\xi_0\in \Upsilon_M$ such equation \eqref{Characteristic_Polynomial_xi}  has three real roots, one of them with double multiplicity, $\lambda_2(\xi_0)=\lambda_3(\xi_0)$ (according to Proposition \ref{descriptioneig} that is when $|\xi_0|=0$, or ${\tau }/{\beta}\neq {1}/{9}$ and $|\xi_0|=\sqrt{m_1} \textrm{ or }\sqrt{m_2}$). In this case, the solution of \eqref{Bose_Problem_Linear_Fourier} is given by
\begin{equation}\label{Solution_Formula_3}
\hat{u}(\xi_0,t)=C_1(\xi_0)e^{\lambda_1(\xi_0)t}+\left(C_2(\xi_0) +C_3(\xi_0)t\right)e^{\lambda_2(\xi_0)t},
\end{equation}
where $C_1,\,C_2$ and $C_3$ are the solutions of the system (again obtained imposing the corresponding initial conditions of $\hat{u}(\xi_0,t)$):
\begin{equation}\label{C_system_3}
\left\{
\begin{array}{ll}
C_1+C_2=\hat{u}_0,\vspace{0.2cm } \\
\lambda_1 C_1 +  \lambda_2 C_2 + C_3 = \hat{u}_1,\vspace{0.2cm } \\
\lambda_1^2 C_1 +  \lambda_2^2 C_2 +2\lambda_2C_3= \hat{u}_2 .
\end{array}
\right.
\end{equation}
The same estimates as in \eqref{C_i_Estimate} hold. Consequently,
we obtain
\begin{equation*}
|\hat{u}(\xi_0,t)|\leq C_M(1+t)\left(|\hat{u}_0(\xi_0)|+ |\hat{u}_1(\xi_0)|+|\hat{u}_2(\xi_0)|\right)e^{-c_4 t}.
\end{equation*}
%

Finally, suppose that we are in the special case of $\xi_0\in \Upsilon_M$ such that equation \eqref{Characteristic_Polynomial_xi} has a triple real root $\lambda_{1,2,3}(\xi_0)= \lambda(\xi_0)$. According to Proposition \ref{descriptioneig} this would happen when $\tau /\beta=1/9$ and $|\xi_0|=\sqrt{m_1}=\sqrt{m_2}$ and we would have $\lambda(\xi_0)=-3/\beta$. In this case, the solution of \eqref{Bose_Problem_Linear_Fourier} is given by
\begin{equation}\label{Solution_Formula_4}
\hat{u}(\xi_0,t)=\left(C_1(\xi_0) +C_2(\xi_0) t +C_3(\xi_0)t^2\right)e^{-\frac{3}{\beta}t},
\end{equation}
where $C_1,\,C_2$ and $C_3$ are the solutions of the system obtained by imposing the corresponding initial conditions of $\hat{u}(\xi_0,t)$:
\begin{equation}\label{C_system_4}
\left\{
\begin{array}{ll}
C_1=\hat{u}_0,\vspace{0.2cm } \\
\lambda C_1 +  C_2  = \hat{u}_1,\vspace{0.2cm } \\
\lambda^2 C_1 +  2 \lambda C_2 +2 C_3= \hat{u}_2 .
\end{array}
\right.
\end{equation}
The same estimates as in the previous cases hold and, hence, we obtain
\begin{equation*}
|\hat{u}(\xi_0,t)|\leq C_M(1+t+t^2)\left(|\hat{u}_0(\xi_0)|+ |\hat{u}_1(\xi_0)|+|\hat{u}_2(\xi_0)|\right)e^{-\frac{3}{\beta} t}.
\end{equation*}

\end{proof}
\section{Decay estimates}\label{sec:decay}
In this section, we show the decay estimates for the solution $u(x,t)$ of the system  \eqref{Bose_Problem}-\eqref{Initial_data_1} using the eigenvalues expansion results of Section \ref{sec:eigenvalues}. For simplicity in the notation, all the norms of this section will be in $\R^N$, although we are going to skip this in the notation.

\begin{theorem}[$L^1$--initial data]\label{Theorem_Solution_Energy_Space}
Let $(u_0, u_1,u_2)\in L^1(\R^N)\cap H^s(\R^N)$, $s\geq 1$, and $0<\tau<\beta$. Then for any $t\geq 0$ the following decay estimates hold for all $0\leq j\leq s$ and certain constants $C,c>0$ independent of $t$ and of the initial data:
\begin{eqnarray}\label{Estimate_N_2}
\left\Vert \partial _{x}^{j}u\left( t\right) \right\Vert _{L^{2}}&\leq& C(\Vert u_0 \Vert_{L^1}+\Vert u_1 \Vert_{L^1}+\Vert u_2\Vert_{L^1})(1+t)^{1-N/4-j/2}\notag\\
&&+C(\Vert\partial_x^j  u_0\Vert_{L^2}+\Vert\partial_x^j  u_1\Vert_{L^2}+\Vert \partial_x^j u_2\Vert_{L^2})e^{-ct}.
\end{eqnarray}
where  $c=\min\left\{ \frac{1}{\beta},\frac{\beta-\tau }{2\beta\tau }, |\Re(\lambda_{2,3}(\xi_{\nu}))|\right\}$.

\end{theorem}

The above estimates can be improved for if $N+j\geq 3$ and get the following Theorem.
\begin{theorem}\label{Theorem_N_3}
Let $(u_0, u_1,u_2)\in L^1(\R^N)\cap H^s(\R^N)$, $s\geq 1$, and $0<\tau<\beta$. Assume that $N+j\geq 3$.  Then, for $t\geq 0$, the following decay estimates hold for all $0\leq j\leq s$ and certain constants $C,c>0$ independent of $t$ and of the initial data:
\begin{eqnarray}\label{Estimate_N_3}
\left\Vert \partial _{x}^{j}u\left( t\right) \right\Vert _{L^{2}}&\leq& C(\Vert u_0 \Vert_{L^1}+\Vert u_1 \Vert_{L^1}+\Vert u_2\Vert_{L^1})(1+t)^{-(N-2)/4-j/2}\notag\\
&&+C(\Vert\partial_x^j  u_0\Vert_{L^2}+\Vert\partial_x^j  u_1\Vert_{L^2}+\Vert \partial_x^j u_2\Vert_{L^2})e^{-ct}.
\end{eqnarray}
where  $c=\min\left\{ \frac{1}{\beta},\frac{\beta-\tau }{2\beta\tau }, |\Re(\lambda_{2,3}(\xi_{\nu}))|\right\}$.
\end{theorem}
\begin{remark}
It is clear that for $N=3$, the estimate \eqref{Estimate_N_3} improves the one in \eqref{Estimate_N_2}. Indeed, from \eqref{Estimate_N_2}, we deduce in this case that the $L^2$-norm of the solution does not decay if $j=0$. On the other hand and from \eqref{Estimate_N_3} we deduce that, for $N=3$, the $L^2$-norm of the solution decays with the rate $(1+t)^{-1/4}.$ Also, for $N=1$ and $j\geq 2$ or $N=2$ and $j\geq 1$, the estimate \eqref{Estimate_N_3} gives faster decay rate than \eqref{Estimate_N_2}.
\end{remark}

\begin{proof}[Proof of Theorem \ref{Theorem_Solution_Energy_Space}]
Applying the Plancherel theorem, we obtain%
\begin{eqnarray}
\left\Vert \partial _{x}^{j}u\left( t\right) \right\Vert _{L^{2}}^{2}
&=&\int_{\mathbb{R^N} }\vert \xi \vert ^{2j}\vert \hat{u}%
\left( \xi ,t\right) \vert ^{2}d\xi  \notag \\
&= &\int_{\Upsilon_L}\vert \xi \vert ^{2j}\vert \hat{u}%
\left( \xi ,t\right) \vert ^{2}d\xi +\int_{\Upsilon_M}\vert \xi \vert ^{2j}\vert \hat{u}%
\left( \xi ,t\right) \vert ^{2}d\xi+\int_{\Upsilon_H}\vert \xi \vert ^{2j}\vert \hat{u}%
\left( \xi ,t\right) \vert ^{2}d\xi.
  \label{deron_U_equality_Y_L}
\end{eqnarray}
Using Proposition \ref{Proposition_Low} (the estimate \eqref{Low_F_Estimate_1}), we have in the low frequency region
\begin{eqnarray}\label{Low_k_1}
\int_{\Upsilon_L}\vert \xi \vert ^{2j}\vert \hat{u}%
\left( \xi ,t\right) \vert ^{2}d\xi&\leq & C\int_{\Upsilon_L}\vert \xi \vert ^{2j}\left(|\xi|^4|\hat{u}_0|^2+|\xi|^4|\hat{u}_1|^2+|\hat{u}_2|^2\right) e^{-\frac{2}{\tau }t}d\xi\notag\\
&&+C\int_{\Upsilon_L}\vert \xi \vert ^{2j}\left(|\hat{u}_0|^2+|\xi|^4|\hat{u}_1|^2+|\hat{u}_2|^2\right) e^{-(\beta-\tau )|\xi|^2 t}|\cos(|\xi| t)|^2d\xi\notag\\
&&+C\int_{\Upsilon_L}\vert \xi \vert ^{2j} \left( |\xi|^2|\hat{u_0}|^2+\dfrac{1}{|\xi|^2}|\hat{u}_1|^2+\dfrac{1}{|\xi|^2}|\hat{u}_2|^2\right) e^{-(\beta-\tau )|\xi|^2 t}|\sin(|\xi| t)|^2d\xi\notag\\
& =&C( L_1+L_2+L_3),
\end{eqnarray}
where $C$ is a generic constant which may take different values in different places.
In \eqref{Low_k_1}  we have used the algebraic inequality $(x+y+z)^2\leq 3(x^2+y^2+z^2).$

For $L_1$ and since $|\xi|$ is small, we have that
\begin{eqnarray*}
L_1&\leq& \int_{\Upsilon_L}\vert \xi \vert ^{2j}\left(|\xi|^4|\hat{u}_0|^2+|\xi|^4|\hat{u}_1|^2+|\hat{u}_2|^2\right) e^{-\frac{2}{\tau }|\xi|^2t}d\xi \\
&\leq &\Vert \hat{u}_0\Vert_{L^\infty(\R^N)}^2\int_{\Upsilon_L}  \vert \xi \vert ^{2j+4} e^{-\frac{2}{\tau }|\xi|^2t}d\xi+\Vert \hat{u}_1\Vert_{L^\infty(\R^N)}^2\int_{\Upsilon_L}  \vert \xi \vert ^{2j+4} e^{-\frac{2}{\tau }|\xi|^2t}d\xi\notag\\
&&+\Vert \hat{u}_2\Vert_{L^\infty(\R^N)}^2\int_{\Upsilon_L}  \vert \xi \vert ^{2j} e^{-\frac{2}{\tau }|\xi|^2t}d\xi
\end{eqnarray*}

This gives, by using the estimate $\Vert \hat{u}_i\Vert_{L^\infty(\R^N)}\leq \Vert u_i\Vert_{L^1(\R^N)},\, i=0,1,2, $ and the bound in \eqref{Estimate_Integral},
\begin{eqnarray*}
L_1&\leq &C\Vert u_0\Vert_{L^1}^2(1+t)^{-j-2-N/2}+C\Vert u_1\Vert_{L^1}^2(1+t)^{-j-2-N/2}+C\Vert u_2\Vert_{L^1}^2(1+t)^{-j-N/2}\\
&\leq&C(\Vert u_0 \Vert_{L^1}^2+\Vert u_1 \Vert_{L^1}^2+\Vert u_2\Vert_{L^1}^2)(1+t)^{-j-N/2}.
\end{eqnarray*}
For $L_2$ and recalling that $0<\tau <\beta$ (dissipative case), by using the estimate \eqref{Estimate_Cos_N_3} and with the same method as before we have
\begin{eqnarray*}
L_2\leq C(\Vert u_0 \Vert_{L^1}^2+\Vert u_1 \Vert_{L^1}^2+\Vert u_2\Vert_{L^1}^2)(1+t)^{-j-N/2}.
\end{eqnarray*}
For the term $L_3$ and using \eqref{Estimate_Sin_N_2}, we have that
\begin{eqnarray*}
L_3&\leq& C\Vert u_0 \Vert_{L^1}^2 (1+t)^{-N/2-j}+C\Vert u_1 \Vert_{L^1}^2 (1+t)^{2-N/2-j}
+C\Vert u_2 \Vert_{L^1}^2 (1+t)^{2-N/2-j}\\
&\leq& C(\Vert u_0 \Vert_{L^1}^2+\Vert u_1 \Vert_{L^1}^2+\Vert u_2\Vert_{L^1}^2)(1+t)^{2-N/2-j}.
\end{eqnarray*}
Collecting the above estimates, we obtain
\begin{equation}\label{Y_L_Estimate_2}
\int_{\Upsilon_L}\vert \xi \vert ^{2j}\vert \hat{u}%
\left( \xi ,t\right) \vert ^{2}d\xi
\leq C(\Vert u_0 \Vert_{L^1}^2+\Vert u_1 \Vert_{L^1}^2+\Vert u_2\Vert_{L^1}^2)(1+t)^{2-N/2-j}.
\end{equation}
Next, in $\Upsilon_H$, we can use \eqref{High_F_Estimate} and proceed as above, obtaining that

\begin{align}\label{Estimate_H_Main}
&\hspace{1cm}\int_{\Upsilon_H}\vert \xi \vert ^{2j}\vert \hat{u}%
\left( \xi ,t\right) \vert ^{2}d\xi\nonumber\\
&\leq C\int _{\Upsilon_H}\vert \xi \vert ^{2j}\Big((1+|\xi|^{-2}+|\xi|^{-4})|\hat{u}_0(\xi)|^2+(|\xi|^{-2}+|\xi|^{-4})|\hat{u}_1(\xi)|^2+(|\xi|^{-4}+|\xi|^{-6})|\hat{u}_2(\xi)|^2\Big)e^{-2c_3t}d\xi\nonumber\\
&\leq  C\int _{\Upsilon_H}\vert \xi \vert ^{2j}\Big(|\hat{u}_0(\xi)|^2+|\hat{u}_1(\xi)|^2+|\hat{u}_2(\xi)|^2\Big)e^{-2c_3t}d\xi\notag\\
&\leq  C(\Vert\partial_x^j  u_0\Vert_{L^2}+\Vert\partial_x^j  u_1\Vert_{L^2}+\Vert \partial_x^j u_2\Vert_{L^2})e^{-2c_3t}
\end{align}
where we have used the fact that $\xi$ is in $\Upsilon_H$, so all the terms $|\xi|^{-2}, |\xi|^{-4}, |\xi|^{-6}$ are bounded.

Finally, in $\Upsilon_M$, we have, by exploiting  \eqref{Middle_F_Estimate},
\begin{eqnarray}\label{Estimate_M_1}
\int_{\Upsilon_M}\vert \xi \vert ^{2j}\vert \hat{u}%
\left( \xi ,t\right) \vert ^{2}d\xi\leq C(\Vert\partial_x^j  u_0\Vert^2_{L^2}+\Vert \partial_x^j u_1\Vert^2_{L^2}+\Vert \partial_x^j u_2\Vert^2_{L^2})e^{-2c_4t}.
\end{eqnarray}

Consequently,  \eqref{Y_L_Estimate_2}, \eqref{Estimate_H_Main} (together with the Sobolev embedding) and  \eqref{Estimate_M_1}, then the estimate \eqref{Estimate_N_3} and \eqref{Estimate_N_2} hold with $c=\min\{c_3,c_4\}$, hence, $$c=\min\left\{ \frac{1}{\beta},\frac{\beta-\tau }{2\beta\tau }, |\Re(\lambda_{2,3}(\xi_{\nu}))|\right\}.$$ Observe that, as we had already pointed out in the previous Remarks \ref{rem:measurezero1} and \ref{rem:measurezero2}, the values of $\xi$ where we have double or triple real roots are sets of measure zero and, hence, they do not affect the decay of the solution.
\end{proof}
\begin{proof}[Proof of Theorem \ref{Theorem_N_3}]
The proof is exactly the same as the one of Theorem  \ref{Theorem_Solution_Energy_Space}, except for the estimate of $L_3$ in \eqref{Low_k_1}.
To estimate $L_3$ we have, by using \eqref{Estimate_Sin_N_3} and $0<\tau <\beta$ that, for $N+j\geq 3$,
\begin{eqnarray*}
L_3&\leq& C\Vert u_0 \Vert_{L^1}^2 (1+t)^{-N/2-j-1}+C\Vert u_1 \Vert_{L^1}^2 (1+t)^{-(N-2)/2-j} + C\Vert u_2 \Vert_{L^1}^2 (1+t)^{-(N-2)/2-j}\\
&\leq & C(\Vert u_0 \Vert_{L^1}^2+\Vert u_1 \Vert_{L^1}^2+\Vert u_2\Vert_{L^1}^2)(1+t)^{-(N-2)/2-j}.
\end{eqnarray*}
Collecting the above estimate with the estimates of $L_1$ and $L_2$ in the proof of Theorem \ref{Theorem_Solution_Energy_Space}, we obtain
\begin{equation}\label{Y_L_Estimate_1}
\int_{\Upsilon_L}\vert \xi \vert ^{2j}\vert \hat{u}%
\left( \xi ,t\right) \vert ^{2}d\xi
\leq C(\Vert u_0 \Vert_{L^1}^2+\Vert u_1 \Vert_{L^1}^2+\Vert u_2\Vert_{L^1}^2)(1+t)^{-(N-2)/2-j}\quad \text{for}\quad N+j\geq 3.
\end{equation}
As we have said, the estimates in $\Upsilon_M$ and $\Upsilon_H$ remain the same as in the proof of Theorem \ref{Theorem_Solution_Energy_Space}.
\end{proof}

\begin{theorem}[$L^{1,1}$--initial data]\label{Theorem_L_1_1_data}
Let $s\geq 1$, $\tau<\beta$ and let $u_0\in L^{1}(\R^N)\cap H^s(\R^N)$, $(u_1,u_2)\in L^{1,1}(\R^N)\cap H^s(\R^N)$ with $\int_{\R^N} u_i(x)dx=0,\,i=1,2$. Then, for $0\leq j\leq s$, the following decay estimate holds:
\begin{eqnarray}\label{Estimate_N}
\left\Vert \partial _{x}^{j}u\left( t\right) \right\Vert _{L^{2}}&\leq& C(\Vert u_0 \Vert_{L^1}+\Vert u_1 \Vert_{L^{1,1}}+\Vert u_2\Vert_{L^{1,1}})(1+t)^{-N/4-j/2}\notag\\
&+& C(\Vert\partial_x^j  u_0\Vert_{L^2} + \Vert\partial_x^j  u_1\Vert_{L^2} + \Vert \partial_x^j u_2\Vert_{L^2} )e^{-ct}
\end{eqnarray}
where $c= \min\{ c_1,c_3,c_4\}$ (defined in Propositions \ref{Proposition_Low}, \ref{Proposition_high} and Lemma \ref{Lemma_Y_M_2}, respectively).

\end{theorem}

\begin{proof}
The proof is the same as the one of Theorem \ref{Theorem_Solution_Energy_Space}, except the part of the low frequency region.
So, for $\xi\in\Upsilon_L$, we have by making use of \eqref{Low_F_Estimate} that
\begin{eqnarray*}
\int_{\Upsilon_L}\vert \xi \vert ^{2j}\vert \hat{u}%
\left( \xi ,t\right) \vert ^{2}d\xi&\leq & C\int_{\Upsilon_L}\vert \xi \vert ^{2j}(|\xi|^4|\hat{u}_0|^2+|\xi|^4|\hat{u}_1|^2+|\hat{u}_2|^2) e^{-2c_1t} d\xi\\
&+& C\int_{\Upsilon_L}\vert \xi \vert ^{2j}(|\hat{u}_0|^2+|\xi|^4|\hat{u}_1|^2+|\hat{u}_2|^2) e^{-2c_2|\xi|^2 t}\vert\cos (|\xi| t)\vert^2d\xi\\
&+& C\int_{\Upsilon_L}\vert \xi \vert ^{2j}\left( |\xi|^2|\hat{u_0}|^2+\Vert u_1\Vert_{L^{1,1}}^2+\Vert u_2\Vert_{L^{1,1}}^2\right)e^{-2c_2|\xi|^2 t}|\sin (|\xi| t)|^2d\xi\\
&\leq  & C (\Vert u_0\Vert_{L^1}^2+\Vert u_1\Vert_{L^{1,1}}^2+\Vert u_2\Vert_{L^{1,1}}^2)(1+t)^{-N/2-j},
\end{eqnarray*}
where we have used that as $\xi\in\Upsilon_L$ we have $e^{-{2}c_1 t}\leq e^{-{2 c_1|\xi|^2} t}$, that the sinus and cosinus functions are bounded and the inequalities \eqref{L11bound} and \eqref{Estimate_Integral}. Collecting this estimate with the other estimates obtained in the proof of Theorem \ref{Theorem_Solution_Energy_Space} for the high and middle frequency regions and, proceeding in the same way as in the proof of Theorem \ref{Theorem_Solution_Energy_Space},  then \eqref{Estimate_N} is fulfilled.

\end{proof}

\section{Acknowledgments}
The authors would like to thank Prof. Dr. R. Racke  and also Prof. Dr. J. Sol\`a-Morales for their helpful discussions on the problem. 

This work is partially supported by the grants MTM2014-52402-C3-3-P (Spain) and MPC UdG 2016/047 (U. de Girona, Catalonia). Also, M. Pellicer is part of the Catalan research group 2014 SGR 1083.

\end{document}